\documentclass{amsart}
\usepackage{enumitem,mathtools,url,tikz-cd,textcomp,hyperref,
amsfonts,amssymb,amsbsy,amsmath,
pifont,latexsym,ifthen,amsthm,rotating,calc,textcase,booktabs,
graphicx}
\usepackage[errorshow]{tracefnt}

\theoremstyle{definition}
\newtheorem{theorem}{Theorem}[section]
\newtheorem{lemma}[theorem]{Lemma}
\newtheorem{corollary}[theorem]{Corollary}
\newtheorem{proposition}[theorem]{Proposition}

\numberwithin{equation}{section}

\hypersetup{
     colorlinks   = true,
     citecolor    = blue,
     linkcolor    = blue
}
\allowdisplaybreaks
\newcommand{\up}[1]{^{(#1)}}
\newcommand{\case}[1]{\begin{cases}#1\end{cases}}

\newcommand{\A}{\mathbb{A}}
\newcommand{\Q}{\mathbb{Q}}

\newcommand{\Z}{\mathbb{Z}}
\newcommand{\F}{\mathbb{F}}
\newcommand{\C}{\mathbb{C}}

\newcommand{\N}{\mathbb{N}}

\newcommand{\Norm}{\operatorname{N}}

\newcommand{\Tr}{\operatorname{Tr}}

\newcommand{\ov}{\overline}

\newcommand{\Gm}[1]{\mathbb{G}_{{\rm m},#1}}

\newcommand{\Frob}{{\rm Frob}}
\newcommand{\bbP}{\mathbb{P}}

\newtheorem{remark}[theorem]{Remark}
\DeclareMathOperator{\rank}{rank}

\DeclareMathOperator{\ord}{ord}

\DeclareMathOperator{\Char}{char}
\DeclareMathOperator{\Span}{span}

\DeclareMathOperator{\Gal}{Gal}

\DeclareMathOperator{\Spec}{Spec}

\DeclareMathOperator{\Hom}{Hom}

\DeclareMathOperator{\GL}{GL}

\DeclareMathOperator{\codim}{codim}
\DeclarePairedDelimiter{\abs}{\lvert}{\rvert}

\DeclarePairedDelimiter{\angles}{\langle}{\rangle}
\DeclarePairedDelimiter{\parens}{(}{)}

\DeclarePairedDelimiter{\floor}{\lfloor}{\rfloor}
\DeclarePairedDelimiter{\ceil}{\lceil}{\rceil}
\begin{document}

       \address{Department of Mathematics, Indiana University, 831 E. Third St., Bloomington, IN 47405, USA}
       \email{xu56@indiana.edu}
       
\title{Stratification for multiplicative character sums}
%\shorttitle{Stratification for multiplicative character sums}
%\volumeyear{2017}
%\paperID{rnn???}
\author{Junyan Xu}%\affil{1}}
%\abbrevauthor{Junyan Xu}
%\headabbrevauthor{Junyan Xu}

%\address{ \affilnum{1}Department of Mathematics, Indiana University, %Rawles Hall, 
%831 E. Third St., Bloomington, IN 47405, USA}
%\correspdetails{xu56@indiana.edu}

%\received{?? Month 2017}
%\revised{?? Month 2017}
%\accepted{?? Month 2017}

\begin{abstract}
We prove a stratification result for certain families of $n$-dimensional (complete algebraic) multiplicative character sums. The character sums we consider are sums of products of $r$ multiplicative characters evaluated at rational functions, and the families (with $nr$ parameters) are obtained by allowing each of the $r$ rational functions to be replaced by an ``offset", i.e. a translate, of itself. For very general such families, we show that the stratum of the parameter space on which the character sum has maximum weight $n+j$ has codimension at least $j\floor{(r-1)/2(n-1)}$ for $1\le j\le n-1$ and $\ceil{nr/2}$ for $j=n$.
\end{abstract}

\maketitle

%Friendly Notice: use Alt+Left to return after clicking on a hyperlink.

\section{Introduction} 

In this paper we are interested in multiplicative character sums of the following form:
\[S:=\sum_{m\in\kappa^n}\chi_1(F_1(m))\chi_2(F_2(m))\dots\chi_r(F_r(m)),\]
where $\kappa$ is a finite field, $F_i\in\kappa[x_1,\dots,x_n]$, and $\chi_i:\kappa^\times\to\C^\times$ is a multiplicative character (extended to $\kappa$ by stipulating $\chi_i(0)=0$), for each $1\le i\le r$.

It is reasonable to expect square root cancellation for generic polynomials $F_i$, namely, that $\abs{S}\le C(\#\kappa)^{n/2}$ for some constant $C=C(n,r,\{\deg F_i\})$ independent of $\kappa$ for generic choices of the $F_i$'s (with respect to the $\chi_i$'s). 
However, character sums of this form seem difficult 
to deal with, especially if square root cancellation is desired.
One can certainly find a multiplicative character $\chi$ and integers $e_i\ge0$ to write $\chi_i=\chi^{e_i}$, so that $S=\sum_{m\in\kappa^n}\chi(F_1(m)^{e_1}F_2(m)^{e_2}\dots F_r(m)^{e_r})$.
But the square root cancellation result of Katz \cite{nonsing mult} about sums of the form $\sum_m \chi(F(m))$ requires that the homogeneous part of highest degree (the ``leading form") of $F$ defines a nonsingular projective variety, 
which is obviously not the case for our sums as soon as $r>1$ or some $e_i>1$. A generalization of Katz's result by Rojas-Le\'on \cite{sing mult} allows singular leading forms, but the ability to establish square root cancellation is lost with the presence of a single singular point.
A subsequent paper of Rojas-Le\'on \cite{purity2} allows the leading form to be a product of polynomials, but the result applies to additive characters only, and also requires that the factors of the leading form together define a nonsingular variety, among other conditions. 

The present paper confirms that if the $F_i$'s are each allowed to vary independently within an ``offset family" (the family of polynomials $F_i(\,\,\cdot\,+x\up{i})$ parametrized by the ``offset" $x\up{i}\in\kappa^n$),
then for generic members of this family, square root cancellation indeed holds as long as $r\ge 2n-1$.
In fact we are able to obtain a stratification result in the sense of Fouvry and Katz \cite{FouvryKatz}, i.e. to bound the dimensions of the subscheme (the stratum) on which the character sum has maximum weight $n+j$,
for each $1\le j\le n$. Having maximum weight $w$ means being a sum of a bounded number of complex numbers of absolute values $\le(\#\kappa)^{w/2}$, so maximum weight $n$ leads to square root cancellation.
%(Roughly, having weight $w$ means having ``order of magnitude" $(\#\kappa)^{(n+j)/2}$, so that weight $n$ leads to square root cancellation; for the precise definition see Theorem \ref{tr fn}.) 
To formulate the precise statement of our results, we first introduce the following

{\bf Notations, Conventions, and Definitions.}
%\subsection*{Notations, Conventions, and Definitions}
If $\chi$ is a multiplicative character, let $\ord\chi$ denote its order. A rational function $F\in\kappa(x_1,\dots,x_n)$ is called $d$th-power-free if each irreducible factor of $F$ has multiplicity
strictly between $-d$ and $d$. We think of a rational function $F\in\kappa(x_1,\dots,x_n)$ as the quotient of two fixed polynomials $G,H\in\kappa[x_1,\dots,x_n]$, define its degree $\deg F$ as $\max\{\deg G,\deg H\}$, and stipulate that $\chi(F(x))=0$ if $G(x)=0$ or $H(x)=0$, where $x$ is the $n$-tuple $(x_1,\dots,x_n)$. Similarly, we use $x\up i$ to denote an $n$-tuple $(x\up i_1,\dots,x\up i_n)$.

For a subscheme $X\subset\A^{nr}_{\kappa}$, define its degree $\deg X$ to be the degree of its closure in $\bbP^{nr}_{\kappa}$.

Define the constants
\[ \theta_j=\theta_j(n,r):= \case{ ja +\max\{0,b+j-(n-1)\} & \text{if }0\le j\le n-1,\\
\ceil{nr/2} & \text{if }j=n,} \]
if we write $\floor{\frac{r-1}{2}}=(n-1)a+b$ with $a\in\N$ and $0\le b<n-1$.
In particular, $\theta_0=0$, $\theta_1=\floor*{\frac{r-1}{2(n-1)}}$, $\theta_{n-1}=\floor{\frac{r-1}{2}}$, and in general $\theta_j\ge j\floor*{\frac{r-1}{2(n-1)}}$ if $n\ge2$.

A variety in this paper is an integral separated scheme of finite type over a base field, not necessarily algebraically closed.

We now state the main theorem of this paper.
%\subsection*{Statement of the Main Theorem}

\begin{theorem}\label{main}
There exist integers $C,C'\in\N$ and a finite set $\mathcal{S}$ (whose elements are called exceptional primes) that depend on four parameters $n,r,d,D$ such that the following holds.

For each $1\le i\le r$, assume that $d_i:=\ord\chi_i\mid d>0$, let $F_i\in\kappa(x_1,\dots,x_n)$ be a $d_i$th-power-free rational function of degree at most $D$
such that $T_{F_i}:=\{m\in\ov{\kappa}^n\mid F_i(x)\equiv F_i(x+m)\}$ is finite for each $1\le i\le r$, and consider the following family of character sums parametrized by $(x\up{1},\dots,x\up{r})\in\kappa^{nr}$:
\[S(x\up{1},\dots,x\up{r}):=\sum_{m\in\kappa^n}\prod_{i=1}^r\chi_i(F_i(m+x\up{i})).\]
Then whenever ${\rm char\,}\kappa\notin\mathcal{S}$, there exist subschemes $\A^{nr}_\kappa=X_0\supset X_1\supset X_2\supset\dots\supset X_n$, such that the sum of degrees of irreducible components of each $X_j$ is at most $C'$, and such that $\codim X_j\ge\theta_j$ (i.e. $\dim X_j\le nr-\theta_j$) and \[\abs{S(x\up 1,\dots,x\up r)}\le C(\#\kappa)^{(n+j-1)/2}\]
for each $1\le j\le n$ and $(x\up 1,\dots,x\up r)\in\A^{nr}(\kappa)\setminus X_j(\kappa)$. 
\end{theorem}

% $X_1$ plays a key role in intended application in anaNT

The theorem says that square root cancellation holds outside of $X_1$, so $X_1$ is ``the stratum of all exceptional (non-generic) parameter values", 
and $\theta_1=\floor*{\frac{r-1}{2(n-1)}}$ is a lower bound for $\codim X_1$. In particular, we need $r\ge2n-1$ (i.e. an offset family with at least $(2n-1)n$ parameters) to show that square root cancellation holds for generic parameter values (i.e. $\codim X_1>0$).
We shall call a parameter value $(x\up{1},\dots,x\up{r})$ $j$-exceptional if it lies in $X_j(\kappa)$, so that ``exceptional" is the same as ``1-exceptional".

Notice that our assumptions on $F_i$ are very general: they need not actually be polynomials, only rational functions, and no nonsingularity conditions or relations among the $F_i$'s are assumed.
This is due to the generality of the argument: it relies %more/merely 
on the general formalism of $\ell$-adic sheaves and weights as in Weil II \cite{Weil II} but requires no explicit cohomological computations.
In particular, square root cancellation is not established in the usual way by showing that the middle cohomology is pure of weight $n$ and that the higher cohomology groups vanish.
% rough, not fine

An explicit value of the constant $C$ has been obtained by Katz \cite[Theorem 11]{sumBetti} and it does not actually depend on $d$, but we do not know a procedure to explicitly determine $C'$ and $\mathcal{S}$.
It is not clear whether one should expect that better $\theta_j$'s can be obtained for general $F_i$'s, but there should certainly be room for improvement if the $F_i$'s are nice. A na\"ive linear interpolation between $\theta_n=\ceil{\frac{nr}{2}}$ and $\theta_0=0$ yields $\theta_j\approx\frac{jr}{2}$, so that $\lim_{r\to\infty}\frac{\theta_j}{r}=\frac{j}{2}$; this may be a natural goal to aim for. In contrast, with our current $\theta_j$'s the limit is $\frac{j}{2(n-1)}$; in the case $n=2$, this suggests that our result is asymptotically optimal for general $F_i$, though for specific $F_i$'s the situation may be better: in fact, if the $F_i$'s are pairwise non-associate irreducible polynomials and some $\chi_i$ is nontrivial, then $\codim X_n=nr+1$, i.e. there is no $n$-exceptional parameter value at all. 
If we are able to obtain a bound on $\codim X_{n-1}$ for the $T_i$'s (see below) that is better than $\theta_{n-1}=s-1$, a better bound on $\codim X_1$ for $S$ will follow.

%we cannot rule out that for some especially nice $F_i$'s we may even have $\theta_1=\theta_2=r$.

\subsection*{Outline of the proof}
There are three key ingredients of the proof. The first is an elementary transformation which allows us to express the moments over the family of character sums $S$ in terms of $r$ other families of character sums $T_i$, $1\le i\le r$. It is a special case of Lemma \ref{elem}.

\begin{proposition}\label{ing elem}
For $s\in\N$, let $M_\kappa(r,s)$ denote the $2s$-th moment of the character sum $S(x\up1,\dots,x\up r)$ over the parameter space $\kappa^{nr}$. We have
\begin{equation}\label{elem eqn}
M_\kappa(r,s):=\sum_{x\up{1},\dots,x\up{r}\in\kappa^n}\abs{S(x\up{1},\dots,x\up{r})}^{2s}= \sum_{m\up{1},\dots,m\up{2s}\in\kappa^n} \prod_{i=1}^r T_i(m\up{1},\dots,m\up{2s}) 
\end{equation} 
where
\[ T_i(m\up{1},\dots,m\up{2s}):=\sum_{x\in\kappa^n}
\prod_{j=1}^s \chi_i(F_i(m\up{j}+x)) \prod_{j=s+1}^{2s} \chi_i^{-1}(F_i(m\up{j}+x)) =\sum_{x\in\kappa^n}\chi_i(F_{{\bf m}}(x))\]
where
$ F_{{\bf m}}(x):=\prod_{j=1}^s F_i(m\up{j}+x) \prod_{j=s+1}^{2s} F_i(m\up{j}+x)^{-1}.$
\end{proposition}
Normally, $M_\kappa(r,s)/(\#\kappa)^{nr}$ is what is called the moment, but in this paper we call $M_\kappa(r,s)$ the moment for simplicity (to avoid the phrase ``power sum of absolute values"). 
With this terminology, the moments over a subscheme (such as $X_j$) do not exceed the moment $M_\kappa(r,s)$ over the whole parameter space.

Notice that the $T_i$'s are families of character sums of the same form as $S$ but with $2sn$ parameters, so whatever stratification result we prove for general $S$ (as in Theorem \ref{main}) can also be applied to the $T_i$'s, with $r$ replaced by $2s$.

Recall that the family of character sums $S$ has a naturally associated family $S_k$ for each finite extension $k/\kappa$, given by
\[ S_k({\bf x})=S_k(x\up1,\dots,x\up r):= \sum_{m\in k^n} \prod_{i=1}^r \chi_i(\Norm_{k/\kappa}(F_i(m+x\up i))) \]
for ${\bf x}=(x\up 1,\dots,x\up r)\in k^{nr}$.
Let $M_k(r,s):=\sum_{{\bf x}\in k^{nr}}\abs{S_k({\bf x})}^{2s}$ denote the $2s$-th moment of $S_k$. If we replace $\kappa$ by $k$ and $\chi_i$ by $\chi_i\circ\Norm_{k/\kappa}$ in Proposition \ref{ing elem}, we get
\begin{equation}\label {elem eqn+}
M_k(r,s):=\sum_{m\up{1},\dots,m\up{2s}\in k^n} \prod_{i=1}^r T_{i;k}(m\up{1},\dots,m\up{2s}) 
\end{equation}
where
\begin{align*} T_{i;k}(m\up{1},\dots,m\up{2s})&:=\sum_{x\in k^n}
\prod_{j=1}^s \chi_i\circ\Norm_{k/\kappa}(F_i(m\up{j}+x)) \prod_{j=s+1}^{2s} (\chi_i\circ\Norm_{k/\kappa})^{-1}(F_i(m\up{j}+x))\\ &=\sum_{x\in\kappa^n}\chi_i\circ\Norm_{k/\kappa}(F_{{\bf m}}(x)).
\end{align*}\\

The second ingredient connects the moments $M_k(r,s)$ over finite extensions of $k/\kappa$ to the dimensions of the $X_j$'s.
%(sum of Betti numbers) s.t. whenever $\deg F_i\le d$

\begin{proposition}\label{ing tr fn}
%Suppose that $\deg F_i\le d$ for $1\le i\le r$ and let $C=C(d,n,r)$ be Katz's constant mentioned above. 
Let $C,C',\mathcal{S}$ be as in Theorem \ref{main} and assume that $\deg F_i\le D$, $\ord\chi_i\mid d>0$ and $\Char\kappa\notin\mathcal{S}$.
%In the setting of Theorem \ref{main}:
\begin{enumerate}[label={(\alph*)}]
\item If $Y$ be a smooth subvariety of $\A^{nr}_\kappa$ on which the families of character sums $S_k$ are a virtual lisse trace function (see Remark \ref{rmk tr fn}), then for each integer $j$, either 
\begin{enumerate}[label={(\arabic*)}]
\item $\abs{S_k({\bf x})}\le C(\#k)^{(n+j-1)/2}$ for any finite extension $k/\kappa$ and ${\bf x}\in Y(k)$, or
\item  $\displaystyle   
\limsup_{\#k\to\infty} \frac{M_k(r,s)}{(\#k)^{\dim Y}(\#k)^{(n+j)s}}
\ge \limsup_{\#k\to\infty} \frac{\sum_{{\bf x}\in Y(k)}\abs{S_k({\bf x})}^{2s}}{(\#k)^{\dim Y}(\#k)^{(n+j)s}}
\ge1$ for all $s\in\N$.
\end{enumerate}
\item There exists a decomposition of $\A^{nr}_{\kappa}$ into smooth varieties $Y$ such that the sum of their degrees does not exceed $C'$ and the restrictions of $S_k({\bf x})$ to each $Y$ is a virtual lisse trace function.
\end{enumerate}
\end{proposition}

Therefore, for $0\le j\le n$ we may take $X_j$ to be the union of those $Y$ on which the alternative (2) holds, which implies that 
\[\dim X_j\le \max\{\dim Y:Y\text{ satisfies (2)}\}\le \inf_{s\in\N} \limsup_{\#k\to\infty}(\log_{\#k}M_k(r,s)-(n+j)s).\] Upper bounds on $M_k(r,s)$ for all finite extensions $k/\kappa$ thus yield upper bounds on $\dim X_j$ (i.e. lower bounds on $\codim X_j$).

Proposition \ref{ing tr fn}(a) follows from Theorem \ref{tr fn}, and (b) is shown in \S\ref{constr} using Lemma \ref{fam deg} and Lemma \ref{smth dcmpsn}.

The above two ingredients together allow the following bootstrapping process: Starting from bounds on the moments (for all $s$ and all $k/\kappa$), Propsosition \ref{ing tr fn} yields a stratification result (a lower bound on $\codim X_j$ for each $j$). If the bounds are proved for general $S$, we may also apply them to the $T_i$'s. A stratification result for the $T_i$'s in turn yield bounds on the moments of $S$ in the following manner, and the process can then be repeated: write $\A^{nr}=\bigcup_{j=0}^n X_j\setminus X_{j+1}$ (with $X_{n+1}=\varnothing$), apply the respective bounds on $T_{i,k}({\bf x})$ (in place of $S_k({\bf x})$) for ${\bf x}\in X_j(k)\setminus X_{j+1}(k)\subset\A^{nr}(k)\setminus X_{j+1}(k)$, and notice that $\#X_j(k)\le C'(\#k)^{\dim X_j}$ (see Lemma \ref{ing box}). This way we obtain new bounds on the right-hand side of (\ref{elem eqn+}) and hence on the left-hand side $M_k(r,s)$. For details about this process, see \S\ref{sec boot}. 
Starting from the initial input below, each time we run the process, the bounds on the $\codim X_j$'s will be improved, and they tend to certain limits which we call $\theta_j$, and these are the best codimension bounds obtainable by iterated improvement (see \S\ref{sec improve}). \\

The initial input to the iterative bootstrapping process is supplied by the following proposition, the last ingredient of the proof:
\begin{proposition}\label{ing Weil bd}
In the setting of Theorem \ref{main}:
\begin{enumerate}[label={(\alph*)}]
\item The number of parameter values ${\bf m}=(m\up1,\dots,m\up{2s})\in k^{n\cdot2s}$ such that \[F_{{\bf m}}(x):=\prod_{j=1}^s F_i(m\up j+x) \prod_{j=s+1}^{2s} F_i(m\up j+x)^{-1}\] is a perfect $d_i$th power in $\ov{k}(x)=\ov{k}(x_1,\dots,x_n)$, is $O((\#k)^{ns})$ as $k$ varies over finite extensions of $\kappa$.
\item (multivariate Weil bound) If $F_{{\bf m}}\in k(x_1,\dots,x_n)$ is not a perfect $d_i$th power in $\ov{k}(x_1,\dots,x_n)$, then
\[T_{i,k}(m\up1,\dots,m\up{2s}):=\sum_{x\in k^n}\chi_i\circ\Norm_{k/\kappa}(F_{{\bf m}}(x))=O((\#k)^{n-1/2})\] as $k$ varies over finite extensions of $\kappa$.
\end{enumerate}
\end{proposition}

Proposition \ref{ing Weil bd} can be seen to be equivalent to the equality $\codim X_n=ns$ for the sums $T_i$. It was the insight of Michael Larsen that, via the elementary transformation, this rather weak input, the weakest nontrivial bound $O((\#k)^{n-1/2})$ (maximum weight $2n-1$), with square root many exceptions ($\codim X_n=ns$), can be bootstrapped to yield the strongest, square root cancellation bound $O((\#k)^{n/2})$ (maximum weight $n$) for generic parameter values ($\codim X_1>0$). 
This would not work if the exponent in Proposition \ref{ing tr fn}(a)(2) were $(n+j-1)s$ instead of $(n+j)s$, so the integrality of the weights %(which follows from Deligne's work \cite{Weil II}) 
is crucial, since it is exactly the integrality that allows the contrast between $(n+j-1)s$ in (1) and $(n+j)s$ in (2) of \ref{ing tr fn}(a).

Proposition \ref{ing Weil bd}(a) follows from Corollary \ref{power count}, and (b) is proved in Remark \ref{rmk Weil bd}. \\

\subsection*{Number of exceptional values in a box} 
Although we are unable to determine explicitly the subschemes of $\A^{nr}_\kappa$ of exceptional parameter values, we obtain uniform bounds on the sums of the degrees of their irreducible components, and hence are able to bound the number of exceptional values in any box in $\kappa^{nr}$, thanks to the following lemma. This is crucial for our intended application in analytic number theory, which will appear in joint work with Lillian Pierce.
% also $\theta_1>0$ for sufficiently large $r$ is crucial

\begin{lemma}\label{ing box}
Let $X\subset\A^N_\kappa$ be a subscheme of codimension $\theta$ and let $d$ be the sum of the degrees of its irreducible components. If $\{B_i\}_{i=1}^N$ are subsets of $\kappa$, the ``box" $B:=\prod_{i=1}^N B_i$ is naturally a subset of $\A^N(\kappa)$. If $1\le\#B_1\le\#B_2\le\dots\le\#B_n<\infty$, we have
\[\#\parens*{X(\kappa)\cap\prod_{i=1}^N B_i}\le d\prod_{i=\theta+1}^N \#B_i=d(\#B)\prod_{i=1}^\theta (\#B_i)^{-1}.\]
\end{lemma}

For the proof, see Remark \ref{rmk box}. The following is an easy corollary of Theorem \ref{main} and Lemma \ref{ing box} with $N=nr$.

\begin{corollary}\label{main cor}
In the setting of Theorem \ref{main}, if $\{B_i\}_{i=1}^n$ are subsets of $\kappa$ such that $1\le\#B_1\le\#B_2\le\dots\le\#B_n<\infty$, and let $B:=\prod_{i=1}^n B_i\subset\kappa^n$, then
\[\#\{(x\up{1},\dots,x\up{r})\in B^r:\abs{S(x\up{1},\dots,x\up{r})}>C(\#\kappa)^{(n+j-1)/2}\}
\quad\le\quad C'(\#B)^r {\bf b}^{-\theta_j},\]
where ${\bf b}^{-\theta}$ denotes $(\#B_1\#B_2\dots\#B_{n_0})^{-r}(\#B_{n_0+1})^{-\eta}$ if we write $\theta=n_0 r+\eta$ with $n_0\in\N$ and $0\le \eta<r$, so that ${\bf b}^{-\theta_j}=(\#B_1)^{-\theta_j}$ for $0\le j\le n-1$, and 
\[{\bf b}^{-\theta_n}=\case{
(\#B_1\#B_2\dots\#B_{n/2})^{-r}&\text{if }n\text{ is even,}\\
(\#B_1\#B_2\dots\#B_{(n-1)/2})^{-r}(\#B_{(n+1)/2})^{-\ceil{r/2}}&\text{if }n\text{ is odd.}
}\]
\end{corollary}

Now suppose instead that $F_i$ is a $d_i$th-power-free polynomial in $\Z[x_1,\dots,x_n]$ such that $F_i(x+m)\not\equiv F_i(x)$ for all $m\in\Z^n$, or equivalently (Lemma \ref{equiv reduction}), $F_i$ cannot be made independent of $x_1$ by a linear change of coordinates, for each $1\le i\le r$. By Lemma \ref{reduction}, the reductions of $F_i$ modulo almost all (all but finitely many) primes remain $d_i$th-power-free in $\F_p[x_1,\dots,x_n]$ and satisfy $T_{F_i}=\{0\}$.
Therefore, if $\chi_i:\F_p^\times\to\C^\times$ is a multiplicative character of order dividing $d_i$ for each $1\le i\le r$, $\{B_i\}_{i=1}^n$ are subsets of $\F_p$ such that $1\le\#B_1\le\#B_2\le\dots\le\#B_n<\infty$, and $B:=\prod_{i=1}^n B_i$, then by the above corollary,
\[\#\{(x\up{1},\dots,x\up{r})\in B^r:\abs{S(x\up{1},\dots,x\up{r})}>C(\#\kappa)^{(n+j-1)/2}\}
\quad\le\quad C'(\#B)^r {\bf b}^{-\theta_j}\]
for almost all primes $p$ (the finitely many primes in $\mathcal{S}$ also needs to be excluded). A similar result holds when $F_i=G_i/H_i\in\Q(x_1,\dots,x_n)$ is $d_i$th-power-free with $\gcd(G_i,H_i)=1$ and $G_i,H_i\in\Z[x_1,\dots,x_n]$ are not invariant under any translations.

\section{Proof of the Main Theorem}

This section presents a complete proof of Theorem \ref{main} following the outline given in \S1. It relies on some additional lemmas stated and proved in \S3.

\subsection{Construction of the stratification (the $X_j$'s)}\label{constr}

Fix $D, d, n, r\in\N$ where $D$ will be an upper bound for all $\deg F_i$'s and $d>0$ will be a common multiple of all $d_i=\ord \chi_i$. 
Let $\mathcal{P}_0$ be the arithmetic scheme that parametrizes all finite fields $\kappa$ and pairs of polynomials $(G_1,H_1),\dots,(G_r,H_r)$ of degrees $\le D$ with $H_i\not\equiv0$, which is an open subvariety of the affine space over $\Z$ of relative dimension $2r{D+n\choose n}$.
Let $\zeta_d$ be a primitive $d$th root of unity and let $R:=R_d=\Z[1/d,\zeta_d]$.
Then $\Gm R\to\Gm R$ defined by $x\mapsto x^d$ is a cyclic \'etale covering of degree $d$, hence induces a continuous surjective homomorphism $\pi_1(\Gm R)\to\Z/d\Z$. If we let $\ell$ be a prime dividing $d$ and compose this with a homomorphism $\Z/d\Z\to\ov{\Q_\ell}^\times$ sending $1$ to $\zeta_d\in\ov{\Q_\ell}$,
we get a 1-dimensional continuous $\ov{\Q_\ell}$-representation of $\pi_1(\Gm R)$, and hence a pure lisse $\ov{\Q_\ell}$-sheaf of weight 0 and rank 1 on $\Gm R$, denoted $\mathcal{L}_d$. For every $\frak{p}\in\Spec R$, the trace function of $\mathcal{L}_d|_{\Gm {{\rm k}(\frak{p})}}$ is a multiplicative character $\chi_d$ of degree $d$ of the residue field ${\rm k}(\frak{p})$.

If $\kappa$ is a finite field that admits multiplicative characters $\chi_1,\dots,\chi_r$ of orders $d_1,\dots,d_r$ respectively, and $d_i\mid d$ for all $i$, then $\kappa$ is a finite extension of ${\rm k}(\frak{p})$ for any $\frak{p}\in\Spec R$ lying above $(\text{char }\kappa)\in\Spec\Z$,
%maybe say why
 and $\chi_d\circ{\rm N}_{\kappa/{\rm k}(\frak{p})}$ is a multiplicative character of order $d$ of $\kappa$, so $\chi_1,\dots,\chi_r$ are all powers of $\chi_d\circ{\rm N}_{\kappa/{\rm k}(\frak{p})}$.
Let $\mathcal{P}$ be the disjoint union of $\mathcal{P}_{d,e_1,\dots,e_r}$ over all $0\le e_i<d$, where $\mathcal{P}_{d,e_1,\dots,e_r}$ is a copy of $\mathcal{P}_0\times_{\Spec\Z}\Spec R_d$ for each $e_1,\dots,e_r$. For each $1\le i\le r$, consider the ``translate and evaluate" maps $g_i$ and $h_i$ 
which are morphisms $\A^{n+nr}_\mathcal{P}=\A^{n+nr}_\Z\times_{\Spec\Z}\mathcal{P}\to \A^1_\Z$ defined by 
\[((m,x\up1,\dots,x\up r),(G_1,H_1,\dots,G_r,H_r))\mapsto G_i(m+x\up i)\text{ and }H_i(m+x\up i)\]
respectively on $\A_{\mathcal{P}_{d,e_1,\dots,e_r}}^{n+nr}\subset\A_{\mathcal{P}}^{n+nr}$.
Consider $\Gm \Z\subset\A^1_\Z$ and its inverse images under the evaluation maps, and define $U:=\bigcap_{i=1}^r g^{-1}_i(\Gm\Z)\cap h^{-1}_i(\Gm\Z)$, an open dense subscheme of $\A^{n+nr}_\mathcal{P}$.
On the connected component $U_{d,e_1,\dots,e_r}:=U\cap\A^{n+nr}_{\mathcal{P}_{d,e_1,\dots,e_r}}$ of $U$, the maps $g_i$ and $h_i$ factor through $\Gm{R_d}$, and we define a sheaf $\mathcal{L}$ on $U$ by specifying
\[ \mathcal{L}|_{U_{d,e_1,\dots,e_r}}:= \bigotimes_{i=1}^r g^*_i\mathcal{L}_d^{\otimes e_i}\otimes h^*_i\mathcal{L}_d^{\otimes -e_i}.\]
Then for any finite field $\kappa$ and multiplicative characters $\chi_i:\kappa^\times\to\C^\times$ with $\ord \chi_i\mid d$ and rational functions $F_i\in\kappa(x_1,\dots,x_n)$ of degrees $\le D$, if we write $\chi_i=\chi_d^{e_i}$, then there exists a closed point $P=(F_1,\dots,F_r)\in \mathcal{P}_{d,e_1,\dots,e_r}\subset\mathcal{P}$ such that the trace function of $\mathcal{L}$ on the fiber $U\cap\A^{n+nr}_P$ at a point $(m,x\up1,\dots,x\up r)$ equals $\prod_{i=1}^r \chi_i(F_i(m+x\up i))$.
If we now consider the projection $\pi:U\to \A^{nr}_\mathcal{P}$, then the trace function of the complex $\mathcal{K}:=R\pi_!\mathcal{L}$ on $\A^{nr}_P$ gives rise to the family of character sums that we are interested in:
\[ \Tr( \Frob_k \mid \mathcal{K}_{{\bf x}} ) = S_k(x\up1,\dots,x\up r) = \sum_{m\in k^n} \prod_{i=1}^r \chi_i({\rm N}_{k/\kappa}(F_i(m+x\up i))) \]
for any finite extension $k/\kappa$ and ${\bf x}=(x\up1,\dots,x\up r)\in\A^{nr}_P(k)\cong k^{nr}$.

The trace function of $\mathcal{K}$ is the same as that of the alternating sum of its cohomology sheaves $R^j\pi_!\mathcal{L}$, which are constructible mixed sheaves of integer weights (possibly away from finitely many primes), since $\mathcal{L}$ is mixed of integer weights and constructible (in fact pure of weight 0 and lisse); see \cite[Theorem I.9.3]{KW}, \cite[Lemme 6.1.3]{Weil II}, and \cite[Th. finitude, Corollarie 1.5]{SGA4.5}.
%and II.7.7]{KW}. 
Mixed sheaves are iterated extensions of pure sheaves, and the trace function of the mixed sheaf is simply the sum of the trace functions of its pure factors. There exists a decomposition of $\A^{nr}_\mathcal{P}$ into finitely many (locally closed) subschemes: $\A^{nr}_\mathcal{P}=\bigcup_{X\in\mathcal{X}}X$, such that the restrictions of these constructible pure factors to each $X\in\mathcal{X}$ are lisse, 
%Freitag-Kiehl, p.124, Prop.12.10, p.131
so that $S_k({\bf x})$ is a virtual lisse trace function on each $X$ (see Remark \ref{rmk tr fn}).
Moreover, using Lemma \ref{smth dcmpsn}, we may assume that $\pi_\mathcal{P}|_X:X\to\ov{\pi_\mathcal{P}(X)}$, where $\pi_\mathcal{P}:\A^{nr}_\mathcal{P}\to\mathcal{P}$ is the structural morphism, is smooth for each $X$ if we work away from finitely many primes, so that every fiber of $\pi_0|_X$ is smooth over the residue field (a finite field). We may also assume that each $X\in\mathcal{X}$ is connected. By Lemma \ref{fam deg} applied to $\ov{X}$, the closure of $X$ in $\bbP^{nr}_{\ov{\pi_\mathcal{P}(X)}}$, the geometric fibers of $\pi_\mathcal{P}|_X$ are equidimensional of degree no more than $\deg X:=\deg \ov{X}$. We then define $C':=\sum_{X\in\mathcal{X}} \deg X$.

Once we obtain the uniform stratification, we now work one fiber at a time, i.e. we restrict to a closed point $P\in\mathcal{P}$ parametrizing a particular choice of $F_1,\dots,F_r,\chi_1,\dots,\chi_r$ such that $F_i\in\kappa(x_1,\dots,x_n)$ is $d_i$th-power-free, where $\kappa:={\rm k}(P)$. For every $X\in\mathcal{X}$, every connected component $Y$ of the fiber $X_P$ of $X$ over $P$ is a smooth variety over $\kappa$, and $S_k({\bf x})$, the trace function of $\mathcal{K}$ on $Y$, is a virtual lisse trace function (see Theorem \ref{tr fn}).
Let $X_j$ be the union of all $Y$ on which $S_k({\bf x})$ satisfies the alternative (2) in Proposition \ref{ing tr fn}(a) (i.e. has maximum weight $\ge n+j$). Then on the other $Y$, the $S_k({\bf x})$ satisfies the alternative (1) (i.e. has maximum weight $\le n+j-1$), and the union of these $Y$ contains $\A^{nr}_P\setminus X_j$, so \[\abs{S(x\up1,\dots,x\up r)}\le C(\#k)^{(n+j-1)/2}\text{\qquad for all }(x\up1,\dots,x\up r)\in\A^{nr}(\kappa)\setminus X_j(\kappa),\] where $C$ is the sum of the ranks of the lisse sheaves (which is bounded by the sum of the maximal ranks of the cohomology sheaves $R^j\pi_!\mathcal{L}$, which is bounded by Katz's constant).
It is clear the sum of the degrees of the irreducible components of $X_j$ does not exceed $C'$. We have thus proved Proposition \ref{ing tr fn}(b).

\subsection{The bootstrapping process}\label{sec boot}

The setting of the bootstrapping process is as follows. We have a family $S$ of character sums, and for each $s\in\N$ and each $1\le i\le r$ we have the family $T_i$ of character sums obtained from the elementary transformation (\ref{elem eqn}). For the family $S$, we consider the filtration $\A^{nr}=X_0\supset X_1\supset X_2\supset \dots\supset X_n$, where $X_j$, $1\le j\le n$, is the union of smooth varieties on which the maximum weight of $S$ is at least $n+j$. The stratification associated to the filtration consists of the $X_j\setminus X_{j+1}$ (on which $S_k({\bf x})$ has maximum weight exactly $n+j$). Similarly, let $\A^{n\cdot2s}\supset Y_1\supset Y_2\supset\dots\supset Y_n$ be the combined stratification of the $T_i$'s, so that $Y_j$, $1\le j\le n$, is the union of smooth varieties on which the maximum weight of some $T_i$ is at least $n+j$.
Define
\begin{align*}
    c(r,j)&:=\codim X_j=nr-\dim X_j,\\
    \text{and\qquad} c'(2s,j)&:=\codim Y_j=n\cdot2s-\dim Y_j.
\end{align*}

Denote the $2s$-th moment of $S_k$ by $M_k(r,s)$, and define \[m(r,s):=\limsup_{\#k\to\infty}\,(\log_{\#k} M_k(r,s)-nr-ns).\] 

The bootstrapping process relies on following three inequalities: 
\begin{lemma}\label{boot}
\begin{enumerate}[label={(\arabic*)}]
    \item $ c(r,j) \ge \max_{s\in\N}\, (js-\floor{m(r,s)})$ ;
    \item $ m(r,s) \le \max_{0\le j\le n}\, ( js-c(r,j) )$ ;
    \item $ m(r,s) \le ns-nr/2+\max_{0\le j\le n}\, (jr/2-c'(2s,j)) $.
\end{enumerate}
\end{lemma}
\begin{remark}
It can be shown using Theorem \ref{tr fn} that we actually have equality in (2). 
%antitone Galois connection on space of functions that are bounds for $c(r,-)$ and $m(r,-)$ respectively. adjoint?  closure operator.. switching role introduce more sophisticated operation: need infinite application to converge/stabilize
Therefore, $m(r,-)$ %(as a function of $s$) 
can be seen as a ``discrete Legendre transform" of $c(r,-)$, % (as a function of $j$), 
so $m(r,s)$ is a convex function of $s$. (We do not know whether $c(r,j)=\codim X_j$ is a convex function of $j$, but the bounds we get from inequality (1) will always be convex.)
Applying (1) and then (2) (or vice versa) is an (idempotent) closure operator coming from a Galois connection specified by the right-hand sides of both inequalities.

Inequality (3) is a version of (2) with the role of $r$ and $2s$ switched (together with $X_j$ and $Y_j$).
Since the $T_i$'s are also of the form of $S$, 
any universal bound on $c(r,j)=\codim X_j$, in the sense that it holds for all sums of the form $S$ in Theorem \ref{main} (for fixed $n$), also applies to $c'(2s,j)=\codim Y_j$ if we simply replace $r$ by $2s$. Thus we can apply (1) and (3) alternately and repeatedly, which is what we refer to as bootstrapping and what we do in the next subsection.

The crucial point is that (3) has the power of breaking %the loop by breaking 
convexity and idempotency, because $r$ and $2s$ are switched: the bounds on $m(r,s)$ that we get from (1), which are convex in $r$, are usually not convex in $s$, and exactly this gives room for improvement.
In fact the iterated improvement process goes on forever; see Lemma \ref{iter}. The limit bound for $c(r,j)=\codim X_j$ will turn out to be  $\theta_j=\theta_j(n,r)$.

In reality, we do not actually compute the intermediate bounds we get during the iterative bootstrapping process, but instead use (1) and (3) repeatedly to  first obtain the limiting bound on $c(r,n-1)$, and then show that the bounds on all $c(r,j)$ we get after bootstrapping one more time is the best we can get. For details, see \S\ref{sec improve}.
% clever way
%comparison between cases
%easier to compare in the limit
\end{remark}
%closure property vs (Due to?) convexity

\begin{proof}
(1) Since $X_j$ is the union of smooth varieties on which the alternative (2) in Proposition \ref{ing tr fn}(a) holds, and since $\dim X_j$ is the maximum of the dimensions of these smooth varieties, we have
\[ \limsup_{\#k\to\infty} \frac{M_k(r,s)}{(\#k)^{\dim X}(\#k)^{(n+j)s}}  \ge  \limsup_{\#k\to\infty} \frac {\sum_{x\in X_j(k)}\abs{S_k(x)}^{2s}} {(\#k)^{\dim X}(\#k)^{(n+j)s}} \ge 1>0. \]
Taking the logarithm, we see that
\[ \limsup_{\#k\to\infty} \,(\log\#k) \parens*{ \frac{\log M_k(r,s)}{\log \#k} - (\dim X_j+(n+j)s)} >-\infty. \]
Since $\log\#k\to\infty$ as $\#k\to\infty$, we must have 
\[\limsup_{\#k\to\infty} \parens*{ \log_{\#k}\,M_k(r,s) - (\dim X_j+(n+j)s)} \ge 0 \] and hence
\begin{align*} m(r,s)&=\limsup_{\#k\to\infty}\,(\log_{\#k} M_k(r,s)-nr-ns) \\
&\ge \dim X_j+(n+j)s -nr-ns \\ & = js - \codim X_j. \end{align*}
Thus $c(r,j)=\codim X_j\ge js-m(r,s)$ after rearranging, so $c(r,j)\ge js-\floor{m(r,s)}$ because $c(r,j)$ is an integer.

(2) Consider the decomposition $\A^{nr}=\bigcup_{0\le j\le n} X_j\setminus X_{j+1}$, with $X_{n+1}=\varnothing$, and recall that $S_k(x)=O((\#k)^{(n+j)/2})$ as $k$ varies for $x\in X_j(k)\setminus X_{j+1}(k)\subset \A^{nr}(k)\setminus X_{j+1}(k)$. Moreover, $\#X_j(k)=O((\#k)^{\dim X_j})$ as $k$ varies.
Therefore
\begin{align*}
     M_k(r,s) & = \sum_{x\in\A^{nr}(k)} \abs{S_k(x)}^{2s} 
     = \sum_{j=0}^n \sum_{x\in X_j(k)\setminus X_{j+1}(k)} \abs{S_k(x)}^{2s} \\
     & = \sum_{j=0}^n O((\#k)^{\dim X_j})O((\#k)^{(n+j)s}),
     %&=O((\#k)^{\max_{0\le j\le n}\,(\dim X_j+(n+j)s)}),
\end{align*} 
so
\[m(r,s)\le\max_{0\le j\le n}\,(\dim X_j+(n+j)s-nr-ns)=\max_{0\le j\le n}\,(js-\codim X_j) .\]  %.. (Problem: some stratum may be empty! but then $\dim X_j=-\infty$ ..)

(3) Consider the decomposition $\A^{n\cdot2s}=\bigcup_{0\le j\le n}Y_j\setminus Y_{j+1}$, with $Y_{n+1}=\varnothing$, then $T_{i;k}(m)=O((\#k)^{(n+j)/2})$ as $k$ varies for all $1\le i\le r$ and $m\in Y_j(k)\setminus Y_{j+1}(k)=\A^{n\cdot2s}(k)\setminus Y_{j+1}(k)$. Moreover, $\#Y_j(k)=O((\#k)^{\dim Y_j})$ as $k$ varies. Therefore
\begin{align*}
     M_k(r,s)&=\sum_{m\in\A^{n\cdot2s}(k)}\prod_{i=1}^r T_{i;k}(m) = \sum_{j=0}^n \sum_{m\in Y_j(k)\setminus Y_{j+1}(k)} \prod_{i=1}^r T_{i;k}(m)\\
     &= \sum_{j=0}^n O((\#k)^{\dim Y_j}) O((\#k)^{(n+j)r/2}), 
\end{align*}
which yields 
\begin{align*} 
m(r,s) &\le \max_{0\le j\le n}\,(\dim Y_j+(n+j)r/2-nr-ns) \\
& = \max_{0\le j\le n}\, (   ns-(n-j)r/2-c'(2s,j)) \\
& = ns-nr/2+\max_{0\le j\le n}\, (jr/2-c'(2s,j)).
\end{align*}
\end{proof}

%\subsection{Multivariate Weil bound: the initial input}

\subsection{The initial bound and iterated improvement}\label{sec improve}

In this section we aim to obtain initial bounds for the moments $M_k(r,s)$ to start the bootstrapping process.
Recall from (\ref{elem eqn+})
\[M_k(r,s):=\sum_{m\up{1},\dots,m\up{2s}\in k^n} \prod_{i=1}^r T_{i;k}(m\up{1},\dots,m\up{2s})\]
and from Proposition \ref{ing Weil bd} the Weil bound $T_{i;k}(m\up 1,\dots,m\up{2s})=O((\#k)^{n-1/2})$ for all but $O((\#k)^{ns})$ parameter values $(m\up1,\dots,m\up{2s})$. We apply the trivial bound $(\#k)^n$ to these $O((\#k)^{ns})$ parameter values, which yields
% $O((\#k)^{ns})$ parameter values that make $F$ a perfect power 
%the remaining $\asymp(\#k)^{n\cdot2s}$ values
\begin{align*}% \label{init bd}
M_k(r,s) & = ((\#k)^n)^r O((\#k)^{ns}) + O((\#k)^{n-1/2})^r ((\#k)^n)^{2s}\\
& %= O((\#k)^{\max\{ns+nr,n\cdot2s+(n-1/2)r\}})
=
\case{ O((\#k)^{ns+nr}) & \text{if }s\le r/2n, \\ 
        O((\#k)^{n\cdot2s+(n-1/2)r}) & \text{if }s\ge r/2n.} 
\end{align*}
and therefore
\begin{equation}\label{init bd mrs} m(r,s) \le \case { 0 &\text{if }s\le r/2n, \\
    ns-r/2 &\text{if } s\ge r/2n.}
\end{equation}

%get moment bound which is nontrivial when $s\le r/2n$.

Taking $s=\floor{r/2n}$ in inequality (1) in Lemma \ref{boot}, we have
 \[c(r,j)\ge\max_{s\in\N}\,(js-\floor{m(r,s)})\ge j\floor{r/2n},\]
 so $c'(2s,j)\ge j\floor{s/n}$. Now take $s\ge n\ceil{r/2}$, so that $\max_{0\le j\le n}\,(jr/2-c'(2s,j))$ is achieved at $j=0$, and hence $m(r,s)\le ns-nr/2$ by inequality (3). By inequality (1), we then obtain $c(r,j)\ge \ceil{js-(ns-nr/2)}$. For $j<n$, this bound is trivial as $-s+nr/2\le0$, but when $j=n$ we do get a nontrivial bound $c(r,n)\ge \ceil{nr/2}$, so $\theta_n=\ceil{nr/2}$ is indeed a lower bound for $\codim X_n$, and we have $c'(2s,n)\ge\ceil{n\cdot2s/2}=ns$. \\
 
We first aim to iteratively improve the bound on $c(r,n-1)$. This relies on the following lemma:

\begin{lemma} \label{theta plus}
For any function $\theta$ of the variable $r\in\N$, let $\theta^+$ be the function of $r$ defined by 
\[\theta^+(r)=\max_{s\in\N}\,\min\{\,(n-1)s,\,\ceil{r/2}-s+\theta(2s),\,-s+r\,\}.\]
If $\theta(r)$ is a universal lower bound for $c(r,n-1)$ for all $r$, then $\theta^+(r)$ is also.
\end{lemma}
\begin{proof}
Suppose that we have a universal bound $c(r,n-1)\ge\theta(r)$, then $c'(2s,n-1)\ge\theta(2s)$. Therefore, by inequality (3) in Lemma \ref{boot},
\begin{align*}
     m(r,s) & \le ns-nr/2+\max\{\,nr/2-ns,\,
(n-1)r/2-\theta(2s),\, (n-2)r/2,\, \dots,\, 0r/2\,\} \\ 
    & = \max\{0,\, -r/2+ns-\theta(2s),\, ns-r\}
\end{align*}
where we used the bounds $c'(2s,n)\ge ns$ and the trivial bounds $c'(2s,j)\ge0$ for $j<n-1$. By inequality (1), 
\begin{align*} c(r,n-1) & \ge \max_{s\in\N} \, ((n-1)s-\floor{m(r,s)}) \\
& \ge \max_{s\in\N}\min\{\,(n-1)s, \,\ceil{r/2}-s+\theta(2s),\, -s+r\,\}=\theta^+(r),
\end{align*}
so $\theta^+(r)$ is also a universal lower bound for $c(r,n-1)$.
\end{proof}

\begin{lemma}\label{iter}
Let $\theta\up 0(r):=0$ for all $r$, and define $\theta\up i$ inductively by
$\theta\up{i+1}=(\theta\up i)^+$, so that
\[\theta\up{i+1}(r)=\max_{s\in\N}\,\min\{\,(n-1)s,\,\ceil{r/2}-s+\theta\up i(2s),\,-s+r\,\}.\]
Then $\theta\up i(r)\nearrow\theta\up\infty(r):= \floor{(r-1)/2}=\ceil{r/2}-1$ as $i\to\infty$, for any $r\ge1$ and $n\ge2$. 
\end{lemma}
\begin{proof}
We prove that $\theta\up i(r)\le\ceil{r/2}-1$ inductively. Consider the second term $\ceil{r/2}-s+\theta\up i(2s)$ in the definition of $\theta\up{i+1}(r)$. By induction hypothesis, $\theta\up i(2s)\le\ceil{2s/2}-1=s-1$ and hence $\ceil{r/2}-s+\theta\up i(2s)\le\ceil{r/2}-1$ for all $s\in\N$, thus $\theta\up{i+1}(r)\le\ceil{r/2}-1$.

If $\theta\up i(r)\ge\theta\up j(r)$ for all $r\in\N$, then $\theta\up i(2s)\ge\theta\up j(2s)$ for all $s\in\N$, so from the definition of $\theta\up{i+1}$ it is clear that $\theta\up{i+1}(r)\ge\theta\up{j+1}(r)$. Since clearly $\theta\up1(r)\ge\theta\up0(r)$ for all $r\in\N$, we see that $\theta\up{i+1}(r)\ge\theta\up i(r)$ for all $i$ by induction.

It remains to show that $\lim_{i\to\infty}\theta\up i(r)\ge\ceil{r/2}-1$. It suffices to deal with the case $n=2$, since the $\theta\up i(r)$ for $n>2$ is no smaller than the $\theta\up i(r)$ for $n=2$, as is clear from the inductive definition. When $n=2$, we shall show that 
\[\theta\up i(r)\ge\floor*{\frac{r}{2}\parens*{1-\frac{1}{i+1}}}\] by induction. (In fact equality holds if $r$ is even.) This inequality clearly holds for $i=0$. Assuming that it holds for $\theta\up i$, then \[\theta\up{i+1}(r)\ge\max_{s\in\N}\min\{\,s,\,\ceil*{\frac{r}{2}}-\ceil*{\frac{s}{i+1}},\,-s+r\,\}.\]
If we plot $s$, $\ceil{\frac{r}{2}}-\frac{s}{i+1}$ and $-s+r$ as functions of $s$, it is clear that we should look at the intersection of the first two lines, which corresponds to $s=s_0:=\ceil{\frac{r}{2}}(1-\frac{1}{i+2})$, or rather $s=\floor{s_0}$. Since $s_0= \ceil{\frac{r}{2}}-\frac{s_0}{i+1}$, we have \[\floor{s_0}=\floor*{\ceil*{\frac{r}{2}}-\frac{s_0}{i+1}}=\ceil*{\frac{r}{2}}-\ceil*{\frac{s_0}{i+1}}\le\ceil*{\frac{r}{2}}-\ceil*{\frac{\floor{s_0}}{i+1}}.\] Since $s_0<\ceil{\frac{r}{2}}$, we have $\floor{s_0}<\frac{r}{2}$, so $\floor{s_0}< -\floor{s_0}+r$. Therefore at $s=\floor{s_0}$, the minimum of three terms is $\floor{s_0}=\floor*{\ceil{\frac{r}{2}}(1-\frac{1}{i+2})}$, which is no less than $\floor{\frac{r}{2}(1-\frac{1}{i+2})}$, so $\theta\up{i+1}(r)\ge\floor{\frac{r}{2}(1-\frac{1}{i+2})}$. Now \[\lim_{i\to\infty}\theta\up i(r)\ge\theta\up{r-1}(r)\ge\floor*{\frac{r}{2}(1-\frac{1}{r})}=\floor*{\frac{r-1}{2}}.\]
\end{proof}

The function $\theta\up{i+1}$ is obtained from $\theta\up i$ by applying a functional $(\cdot)^+$, and $\theta\up\infty$ is a fixed point of this functional. This functional is monotonic, and this lemma shows that $\theta\up\infty$ is the limiting function obtained from applying the functional repeatedly. It is interesting to note that $n$ does not affect the limiting value (though for $n\ge3$ the convergence becomes exponential), and that we are unable to improve from $\ceil{r/2}-1$ to $\ceil{r/2}$.

Since all $\theta\up i(r)$ are universal lower bounds for $c(r,n-1)$, $\sup_{i\in\N}\theta\up i(r)=\ceil{r/2}-1$ is also a universal lower bound for $c(r,n-1)$. We now use $c(r,n)\ge\ceil{nr/2}$ and $c(r,n-1)\ge\ceil{(r-1)/2}$ to get bounds for all the other $c(r,j)$ ($1\le j\le n-2$). With this improved bound for $c(r,n-1)$, the bound for $m(r,s)$ in the proof of Lemma \ref{theta plus} becomes
\begin{equation}\label{c(r,n-1)}
m(r,s) \le \max\{\,0,(n-1)s-r/2+1,\,ns-r\,\}, \end{equation}
hence by inequality (1)
\[ c(r,j) \ge \max_{s\in\N}\min\{\, js,\, (j-n+1)s+\ceil{r/2}-1,\, (j-n)s+r\,\}. \]
Again, we look at $s=s_0:=(\ceil{\frac{r}{2}}-1)/(n-1)$ where the first two terms are equal. Clearly, the maximum
\[\max_{s\in\N}\min\{\, js,\, (j-n+1)s+\ceil{r/2}-1\,\} \]
is achieved at $s=\floor{s_0}$ or $s=\ceil{s_0}$ if $j<n$, and hence it is equal to $\max\{\,j\floor{s_0},\,-(n-j-1)\ceil{s_0}+\ceil{\frac{r}{2}}-1\,\}$. The third term $(j-n)s+r$ is greater than the first two terms both at $\floor{s_0}$ and at $\ceil{s_0}$, so it does not play a role: indeed, $(j-n)\ceil{s_0}+r>(j-n+1)\ceil{s_0}+\ceil{\frac{r}{2}}-1$ because $\ceil{s_0}\le\ceil{\frac{r}{2}}-1<\floor{\frac{r}{2}}+1=r-(\ceil{\frac{r}{2}}+1)$, so $(j-n)\floor{s_0}+r>(j-n+1)\floor{s_0}+\ceil{\frac{r}{2}}-1\ge j\floor{s_0}$. Writing $\floor{\frac{r-1}{2}}=\ceil{\frac{r}{2}}-1=(n-1)a+b$ with $a\in\N$ and $0\le b<n-1$, it is then easy to work out 
\begin{align*}
\theta_j=\theta_j(n,r)& :=\max\{\,j\floor{s_0},\,-(n-j-1)\ceil{s_0}+\ceil*{\frac{r-1}{2}}\,\}\\
& = ja + \max\{ 0, b+j-(n-1) \}
\end{align*}
for all $0\le j\le n-1$.
Combined with the bound $c(r,n)\ge\theta_n:=\ceil{nr/2}$ which we proved before, this is exactly what is claimed in Theorem \ref{main}.

If we just apply the bootstrapping process once, we actually already get bounds $\vartheta_j$ such that $\lim_{r\to\infty}\frac{\vartheta_j}{r}=\frac{j}{n}$; with all this complicated iterated improvement business, we only improve this limit to $\frac{j}{n-1}$, and the improvement becomes less and less significant as $n$ increases. However, we really cannot do better than our $\theta_j$'s using the bootstrapping method alone: even if we use $c(r,n)=\theta_n=\ceil{nr/2}$ and the better bounds $c(r,j)\ge \ceil{r/2}-1\ge \theta_j$ for all $1\le j\le n-1$ as input, the only effect is to improve (\ref{c(r,n-1)}) 
to \[m(r,s)\le \max\{0,\,(n-1)s-r/2+1,\,ns-nr/2\},\] i.e. to replace the third term $ns-r$ by the smaller $ns-nr/2$.
But for $j<n$, the arguments above has shown that we get the same result even without the third term, so $\theta_j$ is not improved using this bound for $m(r,s)$. 
For $j=n$, we still get $c(r,n)\ge\max_{s\in\N}\min\{ ns,\, s+\ceil{r/2}-1,\, \ceil{nr/2}\}= \ceil{nr/2}=\theta_n$.

%but for $n=2$ case it's a half-worth of improvement so certainly worth doing.
%all the remaining efforts 

%should make extremely clear how to get nontrivial $\theta_1$...?
%if we does not have $w$ vs $w+1$, then no hope to get nontrivial $\theta_1$
%crucial integrality of weights

\section{Lemmas and their proofs}

\subsection{An elementary transformation}

Burgess \cite[Lemma 2]{Burgess} used a transformation to express moments over a complete family of incomplete character sums in terms of an incomplete family of complete character sums. It has since been used as a routine to obtain Burgess type bounds. A simpler form of the transformation appeared already in \cite{DaEr}. %, as Burgess referred to. 
% also appears Lillian Pierce, private communication and prior works
We generalize this transformation to the situation where the summand is a product of $r$ factors; in our setting, it is used to express the moments of a complete family of complete character sums in terms of $r$ other complete families of complete character sums.

% systematically/routinely/widely used
%Burgess pointed to Davenport and Erd\Hos \cite{DaEr} as the origin of this transformation.
% appeared already in the paper
%It was the insight of M. Larsen that when the summands are products of $r$ terms, can express it using $r$ other families of complete sums $T_i$, $1\le i\le r$.
%nothing stops us from applying it to complete character sums 
%\footnote{Here ``complete" means that the sum is over the whole space $\kappa^n$ and ``incomplete" means that it is over a subset of $\kappa^n$, typically a box. A complete family of incomplete sums allows $(x\up1,\dots,x\up n)$ to vary in the whole $\kappa^{nr}$ but allows $m$ to vary only in a box.}

\begin{lemma} \label{elem}
Let $R$ be a commutative ring and let $\sigma_1,\dots,\sigma_s$ be automorphisms of $R$. Let $M$ and $X$ be sets, and let $f_1,\dots,f_r: M\times X\to R$ be functions. Let $S: X^r\to R$ be the function defined by 
\[ S(x\up 1,\dots,x\up r):=\sum_{m\in M}\prod_{i=1}^r f_i(m,x\up i). \]
Then
\[  \sum_{x\up1,\dots,x\up r\in X} \prod_{j=1}^s S(x\up1,\dots,x\up r)^{\sigma_j}
= \sum_{m\up 1,\dots,m\up s\in M} \prod_{i=1}^r T_i(m\up 1,\dots,m\up s), \]
where
\[ T_i(m\up 1,\dots,m\up s):=\sum_{x\in X}\prod_{j=1}^s f_i(m\up j,x)^{\sigma_j}. \]
\end{lemma}
\begin{remark}
In the case $X=M=\kappa^n$, $R=\C$, if we replace $s$ by $2s$ in this lemma, define $\sigma_j$ to be the trivial automorphism for $1\le j\le s$ and complex conjugation for $s+1\le j\le 2s$, and let $f_i(m,x)=\chi_i(F_i(m+x))$, we get Proposition \ref{ing elem}. If moreover $S$ is of the more specific form of $T_i$, this is an equality between the $2s$th moment of the $2r$-parameter sum and the $2r$th moment of the $2s$-parameter sum.
\end{remark}
\begin{proof} % [{Proof of Lemma \ref{elem}}]
\begin{align*}
& \sum_{x\up1,\dots,x\up r\in X} \prod_{j=1}^s S(m,x\up 1,\dots,x\up r)^{\sigma_j} \\
=\, & \sum_{x\up1,\dots,x\up r\in X} \prod_{j=1}^s \sum_{m\in M} \prod_{i=1}^r f_i(m,x\up i)^{\sigma_j} \\
=\, & \sum_{x\up1,\dots,x\up r\in X} \sum_{m\up 1,\dots,m\up s\in M} \prod_{j=1}^s \prod_{i=1}^r f_i(m\up j,x\up i)^{\sigma_j} &&\text{(distributive law)}\\
=\, & \sum_{m\up 1,\dots,m\up s\in M} \sum_{x\up1,\dots,x\up r\in X} \prod_{i=1}^r \prod_{j=1}^s f_i(m\up j,x\up i)^{\sigma_j} \\
=\, & \sum_{m\up 1,\dots,m\up s\in M} \prod_{i=1}^r \sum_{x\in X} \prod_{j=1}^s f_i(m\up j,x)^{\sigma_j} &&\text{(distributive law)}\\
=\, & \sum_{m\up 1,\dots,m\up s\in M} \prod_{i=1}^r T_i(m\up1,\dots,m\up s).
\end{align*}
\end{proof}

\subsection{Geometric connected components} \label{geom conn}

In this section we review some facts about geometric connectedness, in preparation for the proof of Theorem \ref{tr fn}. If $\kappa$ is a field, let $\kappa^s$ denote its separable algebraic closure and $\ov{\kappa}$ its algebraic closure. Let $X$ be a connected scheme of finite type over $\kappa$. For any extension $k/\kappa$, let $X_k$ denote $X\times_\kappa k$. For any extension $k'/k$, $X_{k'}\to X_k$ induces a surjection $\pi_0(X_{k'})\to\pi_0(X_k)$ on the sets of connected components.

$G:=\Gal(\kappa^s/\kappa)$ acts on $\pi_0(X_{\kappa^s})$, which is identified with $\pi_0(X_{\ov{\kappa}})$ via the bijection $\pi_0(X_{\ov{\kappa}})\to\pi_0(X_{\kappa^s})$ \cite[Tag 0363]{Stacks}.
Let $G_0\trianglelefteq G$ denote the kernel of the action, and let $k_0$ denote the subfield of $\kappa^s$ fixed by $G_0$. 
Since $X$ is of finite type over $\kappa$, $X_{\kappa^s}$ is noetherian, so $\pi_0(X_{\kappa^s})$ is finite. Therefore, the connected components are clopen, $G_0$ is a subgroup of finite index of $G$, and $k_0/\kappa$ is a finite extension. We call $k_0$ the splitting field of $X/\kappa$, since it is the smallest extension of $\kappa$ that ``splits" the geometric connected components of $X/\kappa$ completely.
If $f\in\kappa[x]$ is an irreducible polynomial and $X=\Spec\kappa[x]/(f)$ then $k_0$ is the splitting field of $f$.

The action of $G$ on $\pi_0(X_{\kappa^s})$ is transitive: by \cite[Tag 038B]{Stacks}, the union of each orbit is the inverse image of a closed subset of $X$ under $X_{\kappa^s}\to X$. A partition of $\pi_0(X_{\kappa^s})$ into orbits then yields a partition of $X$ into finitely many nonempty disjoint closed subsets. Since $X$ is connected, there can only be one orbit.

% By \cite[Tag 04PZ]{Stacks}, the whole image is a connected component.

% since base extension of nonempty scheme is nonempty..

\begin{lemma}
Let $k$ be an intermediate field of $\kappa^s/\kappa$. The following are equivalent:

\begin{enumerate}[label={(\arabic*)}]
\item every connected component of $X_k$ is geometrically connected;
\item $\pi_0(X_{\kappa^s})\to\pi_0(X_k)$ is injective (hence bijective);
\item $\Gal(\kappa^s/k)$ acts trivially on $\pi_0(X_{\kappa^s})$;
\item $\Gal(\kappa^s/k)\le G_0$;
\item $k_0\subset k$.
\end{enumerate}
If $k/\kappa$ is Galois, they are also equivalent to
\begin{enumerate}[resume,label={(\arabic*)}]
\item some connected component of $X_k$ is geometrically connected;
\item some fiber of $\pi_0(X_{\kappa^s})\to\pi_0(X_k)$ is a singleton;
\item the action of $\Gal(\kappa^s/k)$ on $\pi_0(X_{\kappa^s})$ has a fixed point.
% \item $\bigcap_{\sigma\in G}\sigma\Gal(\kappa^s/k)\sigma^{-1}\le H$;
% \item $k_0\subset$ the normal closure of $k$.
\end{enumerate}
\end{lemma}
\begin{remark}\label{rmk geom conn}
Since (5)$\implies$(1), that every connected component of $X_{k_0}$ is geometrically connected.
Now suppose that $\kappa$ is a finite field, so any algebraic extension of $\kappa$ is Galois. If $k_0\not\subset k$, no connected component of $X_k$ is geometrically connected, since (6)$\implies$(5); by \cite[Tag 04KV]{Stacks}, $X_k$ has no rational points.
\end{remark}
\begin{proof}

(1)$\iff$(2) and (6)$\iff$(7): if $Y\in\pi_0(X_k)$, the inverse image of $Y$ in $\pi_0(X_{\kappa^s})$ consists of the connected components of $Y_{\kappa^s}$, so it is a singleton iff $Y$ is geometrically connected.

(2)$\iff$(3) and (7)$\iff$(8) follow from \cite[Tag 038D (1)]{Stacks}.

(3)$\iff$(4) by definition of $G_0$. 
(4)$\iff$(5) by Galois theory. 
 %so $\Gal(k/\kappa)$, and hence $\Gal(\kappa^s/\kappa)$,
%acts transitively on $\pi_0(X_k)$. \cite[Exercise ]{Liu}.
(3)$\implies$(8) is trivial.

Now assume that $k/\kappa$ is Galois, so $H:=\Gal(\kappa^s/k)$ is normal in $G$.
%(8)$\implies$(3) follows from the fact that a normal subgroup of a group acting transitively on a set that has a fixed point must fix the whole set.

(8)$\implies$(3): if (8) holds, the stabilizer of some element $Y\in\pi_0(X_{\kappa^s})$ in $G$ contains $H$. Since the stabilizers of elements in the same $G$-orbit are conjugate in $G$, and since $H\trianglelefteq G$, we see that $H$ fixes the $G$-orbit of $Y$. Since $G$ acts transitively, $H$ fixes $\pi_0(X_{\kappa^s})$, i.e. (3) holds.
%the stabilizer of every element contains a subgroup conjugate to $H$ in $G$. Since $k/\kappa$ is Galois, $H$ is normal in $G$, so $H$ stabilizes every element of $\pi_0(X_{\kappa^s})$, i.e. (3) holds.
%\trianglelefteq G$, 
\end{proof}

%lemma shows that $k_0$ is the smallest extension of $\kappa$ that ``splits" the geometric connected components completely.

\subsection{Moments of virtual lisse trace functions}% of $\ell$-adic sheaves}
For an $\ov{\Q_\ell}$-sheaf $\mathcal{F}$ on a scheme $X$ over a finite field $\kappa$ and any finite extension $k/\kappa$, let $f_k:X(k)\to\C$ of $\mathcal{F}$ be defined by
\[f_k(x):=\iota(\Tr(\Frob_k \mid \mathcal{F}_{\ov{x}}))\]
where $\iota$ is a fixed isomorphism from $\ov{\Q_\ell}$ to $\C$, and $\ov{x}$ is a geometric point over $x\in X(k)$. We call the collection $\{f_k\}_{k/\kappa\text{ finite}}$'s for all finite extensions $k/\kappa$ the trace function of $\mathcal{F}$, thought of as a function in variables $k$ and $x$. 

All $\ov{\Q_\ell}$-sheaves appearing in this paper will be pure or mixed with integer weights with respect to any isomorphism $\ov{\Q_\ell}\to\C$, but all the arguments go through if we just fix one isomorphism. For simplicity, we shall talk about purity and mixedness without specifying the isomorphism.

\begin{theorem}\label{tr fn}
Let $X$ be a smooth variety over a finite field $\kappa$, and let $\{\mathcal{F}_i\}_{i=1}^N$ and $\{\mathcal{G}_i\}_{i=1}^{N'}$ be pure lisse $\ov{\Q_\ell}$-sheaves on $X$ (of integer weights). For every finite extension $k/\kappa$, let $(f_i)_k,(g_i)_k:X(k)\to\C$ denote the trace functions of $\mathcal{F}_i$ and $\mathcal{G}_i$ respectively, and let $t_k=\sum_{i=1}^N (f_i)_k-\sum_{i=1}^{N'} (g_i)_k$.
Then for each integer $w\in\Z$, either

(1) $\abs{t_k(x)}\le C(\#k)^{w/2}$ for every finite extension $k/\kappa$ and $x\in X(k)$, where $C=\sum_{i=1}^N\rank(\mathcal{F}_i)+\sum_{i=1}^{N'}\rank(\mathcal{G}_i)$, or 

(2) $\displaystyle\limsup_{\#k\to\infty}\frac{\sum_{x\in X(k)}\abs{t_k(x)}^{2s}}{(\#k)^{\dim X}(\#k)^{(w+1)s}}\ge1$ for all $s\ge1$ or $s=0$, in particular for all $s\in\N$. 
\end{theorem}
\begin{remark}\label{rmk tr fn}
Since the trace functions $t_k(x)$ in the statement of the theorem comes from a formal difference of lisse sheaves, we say that $t_k(x)$ is a ``virtual lisse trace function".

The two alternatives (1) and (2) are clearly mutually exclusive since $\#X(k)\le (\deg X)(\#k)^{\dim X}$ (see Remark \ref{rmk box}). 
We call the smallest $w$ that makes (1) true the maximum weight of the virtual trace function $\{t_k\}_{k/\kappa}$, which is also the largest $w$ such that the irreducible constituents of weight $w$ among the sheaves $\mathcal{F}_i$ and $\mathcal{G}_i$ do not all cancel out.

The theorem relates the cumulative and the pointwise behavior of a virtual lisse trace function. It shows that, although one cannot expect a trace function with maximum weight $>w$ has magnitude exceeding $(\#k)^{w+1}$ at every $k$-point, it indeed has such magnitude on average in terms of its $2s$-moments ($s\ge1$), if the variety is smooth and the sheaves are lisse. 
A result like this may be well-known to experts, but I cannot find a reference.
It is easier to prove if the virtual trace function is an actual trace function, so that no cancellation is possible. 
The $s=0$ case (with the convention $0^0=1$) of (2) can alternatively be obtained by applying the theorem to the constantly 1 trace function, or directly from the Lang--Weil bound. 

This lemma can be extended to the case where all $\mathcal{F}_i$, $\mathcal{G}_i$ are lisse and mixed and $X$ is normal:
%and equidimensional: 
an irreducible mixed lisse sheaf on a normal variety is pure and remains irreducible when restricted to a dense open smooth subvariety, and moreover its isomorphism class is determined by the restriction \cite[Lemma I.2.7 and Theorem I.2.8(3)]{KW}.
\end{remark}

%def of virtual lisse trace fn having domain = disjoint union of $X(k)$ over all $k/\kappa$.
% maybe state Katz's lemme 2.2.2.2
% Theorem I.9.3 of Kiehl-Weissauer in proof .. maybe call an integral virtual trace function??

% Theorem.. is about ``virtual lisse trace functions" which include all families of algebraic character sums over algebraic varieties parametrized by alg vartys restricted to appropriate subvariety.. ! restricted to strata on which sheaves are lisse !

%such a result does not appear in Kiehl-Weissauer

\begin{proof} We first reduce to the case that $X$ is geometrically connected. Let $k_0$ be the splitting field (see \S\ref{geom conn}) of $X/\kappa$.
%, which is the smallest extension of $\kappa$ such that every connected component of $X\times_\kappa k_0$ is geometrically connected 
Let $\{X_j\}_{j=1}^J$ be the connected components of $X\times_\kappa k_0$, and consider the restrictions of $\mathcal{F}_i$ and $\mathcal{G}_i$ to the $X_j$'s. Suppose that the lemma is true for these $X_j$'s, which are geometrically connected (Remark \ref{rmk geom conn}). If (1) holds for all of the $X_j$'s, then (1) holds for $X\times_\kappa k_0=\bigcup_{j=1}^J X_j$, so (1) holds for $X$ if $k_0\subset k$. If $k_0\not\subset k$, (1) is vacuously true, since in that case $X(k)=\varnothing$ (Remark \ref{rmk geom conn}). On the other hand, if (2) holds for some $X_j$, then (2) holds for $X\times_\kappa k_0$ since $X_j\subset X\times_\kappa k_0$, hence it holds for $X$ since finite extensions of $k_0$ are also finite extensions of $\kappa$.

Thus we may assume that $X$ is geometrically connected. We then have $\#X(k)=(\#k)^{\dim X}+o((\#k)^{\dim X})$ as $\#k\to\infty$ (Lang--Weil), so we can substitute $\#X(k)$ for $(\#k)^{\dim X}$ in the limsup. Since $x\mapsto x^s$ is convex for $s\ge1$ or $s=0$, by Jensen's inequality,
\[\frac{1}{\#X(k)}\sum_{x\in X(k)}\parens*{\frac{\abs{t_k(x)}^2}{(\#k)^{w+1}}}^s\ge\
\parens*{
\frac{1}{\#X(k)}\sum_{x\in X(k)}\frac{\abs{t_k(x)}^2}{(\#k)^{w+1}}}^s_,\]
thus we see that the $s=1$ case of (2) implies (2) for arbitrary $s\ge1$ or $s=0$. We now focus on the case $s=1$.

Since the trace functions of a lisse sheaf are the sums of the trace functions of its irreducible constituents (with multiplicities), we may assume that all $\mathcal{F}_i$, $\mathcal{G}_i$ are irreducible. Furthermore, we can assume that no $\mathcal{F}_i$ is isomorphic to any $\mathcal{G}_j$, since isomorphic sheaves give rise to identical trace functions which cancel each other.
Let $w_0$ be the maximum weight that appears among the $\mathcal{F}_i$ and $\mathcal{G}_i$. If $w_0\le w$, (1) is true, so we assume that $w_0>w$, and aim to prove the stronger version  %variant 
of (2) with $w+1$ replaced by $w_0$. For this purpose, those $\mathcal{F}_i$ and $\mathcal{G}_i$ with weights $\le w_0$ become irrelevant, since their contribution to the limsup is zero. (When $\abs{t_k(x)}^2$ is expanded, any term that involves a %trace function 
pure lisse sheaf of weight $<w_0$ contributes at most $C^2 (\#k)^{w_0/2}(\#k)^{(w_0-1)/2}=O((\#k)^{w_0-\frac12})$, and $\#X(k)=O((\#k)^{\dim X})$.)

Thus we further assume that all $\mathcal{F}_i$, $\mathcal{G}_i$ are of weight $w_0$. Notice that $f_k:=\sum_{i=1}^N (f_i)_k$ and $g_k:=\sum_{i=1}^{N'} (g_i)_k$ are the trace functions of the semisimple lisse sheaves $\mathcal{F}:=\bigoplus_{i=1}^N\mathcal{F}_i$ and $\mathcal{G}:=\bigoplus_{i=1}^{N'}\mathcal{G}_i$ respectively. Since $\mathcal{F},\mathcal{G}$ have weight $w_0$, the trace functions of the duals $\mathcal{F}^\vee$ and $\mathcal{G}^\vee$ are $\ov{f_k}/(\#k)^{w_0}$ and $\ov{g_k}/(\#k)^{w_0}$ respectively,
so 
\[\abs{t_k}^2/(\#k)^{w_0}=\abs{f_k-g_k}^2/(\#k)^{w_0}=(f_k\ov{f_k}-f_k\ov{g_k}-g_k\ov{f_k}+g_k\ov{g_k})/(\#k)^{w_0}\] is the trace function of $\mathcal{A}:=\mathcal{F}\otimes\mathcal{F}^\vee\oplus\,\mathcal{G}\otimes\mathcal{G}^\vee$ minus that of $\mathcal{B}:=\mathcal{F}\otimes\mathcal{G}^\vee\oplus\,\mathcal{G}\otimes\mathcal{F}^\vee$.

By the Grothendieck--Lefschetz trace formula, 
\[\sum_{x\in X(k)}\abs{t_k(x)}^2/(\#k)^{w_0}=\sum_{j=0}^{2\dim X}\Tr\parens*{\Frob_\kappa^{[k:\kappa]}\mid H^j_c(X_{\ov{\kappa}},\mathcal{A})}-
\Tr\parens*{\Frob_\kappa^{[k:\kappa]}\mid H^j_c(X_{\ov{\kappa}},\mathcal{B})}\]
where $X_{\ov{\kappa}}:=X\times_\kappa\ov{\kappa}$. Since $\mathcal{A}$ and $\mathcal{B}$ are pure of weight 0, the eigenvalues of $\Frob_\kappa^{[k:\kappa]}$ acting on both $H^j_c$ have modulus $\le ((\#\kappa)^{j/2})^{[k:\kappa]}=(\#k)^{j/2}$ (\cite[Theorem 3.3.1]{Weil II}, \cite[Theorem I.7.1]{KW}), so $H_j^c$ contributes zero to the limsup unless $j=2\dim X$. Since $X$ is smooth and geometrically connected, if $\mathcal{H}$ is a lisse $\ov{\Q_\ell}$-sheaf on $X$, 
Poincar\'e duality yields 
\[H^{2\dim X}_c(X_{\ov{\kappa}},\mathcal{H})\cong H^0(X_{\ov{\kappa}},\mathcal{H}^\vee)^\vee(-\dim X)\cong((\mathcal{H}^\vee)^{\pi_1(X_{\ov{\kappa}})})^\vee(-\dim X)\]
as representations of $\pi_1(X)/\pi_1(X_{\ov{\kappa}})\cong\hat{\Z}=\ov{\angles{\Frob_\kappa}}$, so \[\Tr\parens*{\Frob_{\kappa}^{[k:\kappa]}\mid H_c^{2\dim X}(X_{\ov{\kappa}},\mathcal{H})}=(\#k)^{\dim X}\sum_i\lambda_i^{-[k:\kappa]},\] where the $\lambda_i$ are the Frobenius eigenvalues (each appearing as many times as its algebraic multiplicity) on the space of geometric invariants $(\mathcal{H}^\vee)^{\pi_1(X_{\ov{\kappa}})}$. Therefore, if the Frobenius eigenvalues on $(\mathcal{A}^\vee)^{\pi_1(X_{\ov{\kappa}})}\cong\mathcal{A}^{\pi_1(X_{\ov{\kappa}})}$ and $(\mathcal{B}^\vee)^{\pi_1(X_{\ov{\kappa}})}\cong\mathcal{B}^{\pi_1(X_{\ov{\kappa}})}$ are the multi-sets $A$ and $B$ respectively, we have
\[\limsup_{\#k\to\infty}\frac{\sum_{x\in X(k)}\abs{t_k(x)}^2}{(\#k)^{\dim X}(\#k)^{w_0}}=\limsup_{\#k\to\infty}\parens*{\sum_{\lambda\in A}\lambda^{-[k:\kappa]}-\sum_{\lambda\in B}\lambda^{-[k:\kappa]}}.\]
Notice that all these eigenvalues $\lambda$ have modulus 1, since $\mathcal{A}$ and $\mathcal{B}$ are pure of weight 0. Therefore, by \cite[Lemme 2.2.2.2]{Katz-Sommes}, the limsup is at least 1 if $A\neq B$. (In fact, Katz showed that the limsup is at least the square root of the cardinality of the symmetric difference $A\triangle B$ of the multisets $A$ and $B$.)

We now show that $1\in A$ but $1\notin B$.
Recall we assumed that the sheaves (representations) $\mathcal{F}$ and $\mathcal{G}$ are sums of irreducibles, i.e. semisimple. Since duals, tensor products (over the field $\ov{\Q_\ell}$ of characteristic 0), and quotients of semisimple representations are semisimple, we find that $\mathcal{A}^{\pi_1(X_{\ov{\kappa}})}$ and $\mathcal{B}^{\pi_1(X_{\ov{\kappa}})}$ are semisimple $\pi_1(X)$-representations, and since the actions of $\pi_1(X)$ factor through $\hat{\Z}=\ov{\angles{\Frob_\kappa}}$, they are semisimple as $\hat{\Z}$-representations. Since $\Z=\angles{\Frob_\kappa}$ is dense in $\hat{\Z}$ and the action is continuous, $\hat{\Z}$-irreducibles remain irreducible under the $\Z$-action, so $\mathcal{A}^{\pi_1(X_{\ov{\kappa}})}$ and $\mathcal{B}^{\pi_1(X_{\ov{\kappa}})}$ are semisimple  $\Z$-representations, which just means that the action of Frob is diagonalizable. Therefore, the (algebraic) multiplicities of the eigenvalue 1 equal the dimensions of the eigenspaces (the geometric multiplicities). But the eigenspaces associated with eigenvalue 1 simply consist of the elements fixed by $\Frob_\kappa$. Again, since $\angles{\Frob_\kappa}$ is dense in $\hat{\Z}$ through which the actions of $\pi_1(X)$ factor, these eigenspaces are just $\mathcal{A}^{\pi_1(X)}$ and $\mathcal{B}^{\pi_1(X)}$. Since 
\begin{align*}&\mathcal{A}=
\mathcal{F}\otimes\mathcal{F}^\vee\oplus\,\mathcal{G}\otimes\mathcal{G}^\vee\cong\Hom(\mathcal{F,F})\oplus\Hom(\mathcal{G,G})\\
\text{ and \quad}&\mathcal{B}=\mathcal{F}\otimes\mathcal{G}^\vee\oplus\,\mathcal{G}\otimes\mathcal{F}^\vee\cong\Hom(\mathcal{G,F})\oplus\Hom(\mathcal{F,G}),\end{align*} 
we have $\mathcal{A}^{\pi_1(X)}\cong\Hom_{\pi_1(X)}(\mathcal{F,F})\oplus\Hom_{\pi_1(X)}(\mathcal{G,G})\neq0$ (since we assumed that the weight $w_0$ appears in $\mathcal{F}$ or in $\mathcal{G}$, $\mathcal{F}$ and $\mathcal{G}$ cannot both be trivial) and $\mathcal{B}^{\pi_1(X)}=\Hom_{\pi_1(X)}(\mathcal{G,F})\oplus\Hom_{\pi_1(X)}(\mathcal{F,G})=0$ (since we assumed that $\mathcal{F}$ and $\mathcal{G}$ have no common irreducible constituents). Thus $1\in A$ but $1\notin B$, hence $A\neq B$, which completes the proof.
\end{proof}

% going all the way up to algebraic closure => one further purely inseparable extension => irreducible factors remains geometrically connected => power of a single irreducible polynomial .. recall uniq factorization holds for all polynomial ring over any field
\subsection{The multivariate Weil bound}

\begin{lemma}\label{Weil bd}
If $k$ is a finite field, $\chi:k^\times\to\C^\times$ is a multiplicative character of order $d$, and $F\in k(x_1,\dots,x_n)$ is not a perfect $d$th power over $\ov{k}$ (equivalently, $F$ is not of the form $aG^d$ with $a\in k^\times$ and $G\in k(x_1,\dots,x_n)$; see Lemma \ref{reduced}), then
\[\abs*{\sum_{x\in k^n}\chi(F_(x))}\le C(\#k)^{n-1/2},\]
where $C$ depends only on $n$ and $\deg F$.
\end{lemma}
\begin{remark}\label{rmk Weil bd}
This lemma was proved for polynomial $F$ in \cite[Lemma 6]{Shp} without using $\ell$-adic sheaves. 
We use the machinery of Weil II to obtain a proof for rational functions that is more conceptual.
Applying the lemma to $\chi_i\circ\Norm_{k/\kappa}$ which has order equal to $d_i=\ord\chi_i$, we get Proposition \ref{ing Weil bd}(b).
\end{remark}

% weak but powerful
% the only advantage to use rational funtions is that... uniformity of $C$ on degree of rational function. if written as polynomial, degree depends on order of character
% perf $d$th power in alg closure imply perf $d$th power up to unit in base field if perfect field or if $p\nmid d$.

\begin{proof}
Let $d:=\ord\chi$ and let $\ell$ be a prime other than $\Char k$. Let $\mathcal{L}_d$ be the lisse $\ov{\Q_\ell}$-sheaf of weight 0 on $\Gm{k}$ associated to the $\ell$-adic representation $\pi_1(\Gm{k})\to\Z/d\Z\cong \mu_d\subset\ov{\Q_\ell}^\times$ where the first map is associated to the cyclic \'etale covering $\Gm{k}\to\Gm{k}$ defined by $x\mapsto x^d$.

Let $V$ be the open subvariety of $\A^n_k$ on which both the numerator and the denominator of $F$ is nonzero, and let $f:V\to\Gm{k}$ be defined by $F$. % By Remark \ref{rmk lift} below, $f^*\mathcal{L}_d$ is geometrically constant precisely when $F$ is a perfect $d$th power. 
By the Grothendieck--Lefschetz trace formula, \[\sum_{x\in k^n}\chi(F(x))=\sum_{j=0}^{2n}\Tr\parens*{\Frob_k\mid H^j_c(V_{\ov{k}},f^*\mathcal{L}_d)}.\]
If $F$ is not a perfect $d_i$th power over $\ov{k}$, $f^*\mathcal{L}_d$ is not geometrically constant (see \ref{rmk lift} below), and since it is of rank 1, it has no geometric invariants, thus $H^{2n}_c\cong H^0$ vanishes.
Moreover, $H^j_c$ with $j<2n$ has weights $\le j\le 2n-1$ \cite[Theorem 3.3.1]{Weil II} since $f^*\mathcal{L}_d$ is pure of weight 0. Therefore,
\[ \abs*{\Tr\parens*{\Frob_k\mid H^j_c(V_{\ov{k}},f^*\mathcal{L}_d)}} \le \rank(H^j_c(V_{\ov{k}}))\cdot (\#k)^{(2n-1)/2}. \]
Since the ranks of the $H^j_c$ are bounded by Katz's constant $C$ which depends only on $n$ and $\deg F$, we obtain $\abs*{\sum_{x\in k^n}\chi(F(x))}\le C(\#k)^{n-1/2}$.
\end{proof}

\begin{lemma}
Let $X$ and $Z$ be connected schemes, let $G$ be a finite group, and let $\pi_1(X)\to\Gal(Y/X)\cong G$ be a surjective homomorphism associated to a Galois \'etale covering $\varphi:Y\to X$. Let $G\to\GL_N(\ov{\Q_\ell})$ be a faithful representation and let $\mathcal{L}$ denote the $\ov{\Q_\ell}$-sheaf associated to its composition with $\pi_1(X)\to G$. Then for any morphism $f:Z\to X$, $f$ factors through $\varphi$ iff $f^*\mathcal{L}$ is constant.
\end{lemma}
\begin{remark}\label{rmk lift}
If $X=Y=\Gm{\ov{k}}$, $\varphi$ is the $d$th power map, $Z=V_{\ov{k}}$, and $f:Z\to X$ is defined by $F$, then by the lemma we see that 
$F$ is a perfect $d$th power in $\ov{k}(x_1,\dots,x_n)\iff f$ factors through $\varphi \iff f^*\mathcal{L}$ is constant.
\end{remark}
\begin{proof}
$(\implies)$ Notice that $\pi(Y)$ is exactly the kernel of $\pi_1(X)\to G$. If $f$ factors through $\varphi$, the representation associated to $f^*\mathcal{L}$, which is $\pi_1(Z)\to\pi_1(X)\to G\to\GL_N(\ov{\Q_\ell})$, factors through $\pi_1(Y)$ and hence is trivial.

$(\impliedby)$ Consider the following commutative diagram
%\[\begin{tikzcd} \pi_1(Y\times_X Z) \arrow{r} \arrow[swap]{d} & \pi_1(Y) \arrow{d} \\ \pi_1(Z) \arrow{r} & \pi_1(X) \end{tikzcd}\]
\[\begin{tikzcd} Y\times_X Z \arrow{r} \arrow[swap]{d} & Y \arrow{d} \\ Z \arrow{r} & X \end{tikzcd}\]
$\pi_1(Z)$ acts on $Y\times_X Z$ via the action of $\pi_1(X)$ on $Y$. If $f^*\mathcal{L}$ is trivial, $\pi_1(Z)\to\pi_1(X)\to G\cong\Gal(Y/X)$ is trivial because $G\to\GL_N(\ov{\Q_\ell})$ is faithful, so $\pi_1(Z)$ acts trivially on $Y$ and hence on $Y\times_X Z$. Since $Y\times_X Z\to Z$ is an \'etale covering, it must be an isomorphism on every connected component, so in particular it has a section $Z\to Y\times_X Z$. Composing this section with $Y\times_X Z\to Y$ yields a lift $Z\to Y$ of $f:Z\to X$.
\end{proof}

% better $\theta_j$'s, in particular $\theta_1$, can be obtained for a sum $S$ if better $\theta_{n-1}$ for the corresponding $T_i$ are available. However, the results in Rojas-Leon [2] doesn't yield better $\theta_{n-1}$ (singularity dimension $n-2$, nonreduced at every point), though the method may be adaptable to this situation to yield better $\theta_{n-1}$. (may be adapted..?)

\subsection{Mutual transversality of subspaces of a vector space}

\begin{lemma}\label{vec sp}
If $k$ is a field and $\{V_j\}_{j=1}^N$ are $k$-subspaces of a vector spaces $V$ over $k$, then there exists a basis $E$ of $V$ and pairwise disjoint subsets $\{E_j\}_{j=1}^N$ of $E$ such that $\bigcap_{j=1}^N V_j=\Span(E\setminus\bigcup_{j=1}^N E_j)$ and $V_j\subset\Span(E\setminus E_j)$ for $1\le j\le N$.
\end{lemma}
In other words, the $V_j$'s can each be replaced by a larger subspace such that their intersection remain unchanged, so that they are now determined by the vanishing of respective sets of coordinates that are disjoint from each other. The ability to treat these disjoint coordinates separately is important in the proof of Lemma \ref{count}.

\begin{proof}
This lemma follows from Lemma \ref{extend} and Lemma \ref{equiv} ((1)$\implies$(3)) below.
\end{proof}

\begin{lemma}\label{extend}
If $\{V_j\}_{j=1}^N$ are subspaces of a $k$-vector space $V$, then there exist subspaces $\{W_j\}_{j=1}^N$ of $V$ such that $V_j\subset W_j$ and $(\bigcap_{j=1}^n W_j)+W_{n+1}=V$ for $1\le n<N$ and $\bigcap_{j=1}^N V_j=\bigcap_{j=1}^N W_j$.
\end{lemma}

This lemma fails if $V_j$ are finite abelian groups instead of vector spaces, which is the main reason why we cannot extend Lemma \ref{count} to the situation of an abelian group variety acting on another variety, the original situation being a vector space acting simply transitively on the affine space. 
% example : <(2,1)> and <(0,1)> in Z/8 + Z/2

\begin{proof}
We proceed by induction. If $N=0$, there is nothing to prove (the empty intersection is always $V$). If $N>0$, given $\{V_j\}_{j=1}^N$, apply the induction hypothesis to $\{V_j\}_{j=1}^{N-1}$ to get $\{W_j\}_{j=1}^{N-1}$. Let $U$ be a complement of $(\bigcap_{j=1}^{N-1} W_j)+V_N$ in $V$, and let $W_N:=V_N+U\supset V_N$, then clearly $(\bigcap_{j=1}^{N-1} W_j)+W_N=V$. If $A,B,C\subset V$ are subspaces satisfying $(A+B)\cap C=\{0\}$, it is easy to show that $A\cap(B+C)=A\cap B$. Taking $A=\bigcap_{j=1}^{N-1} W_j$, $B=V_N$ and $C=U$, we see that \[\textstyle\bigcap_{j=1}^N W_j=A\cap(B+C)=A\cap B=(\bigcap_{j=1}^{N-1} W_j)\cap V_N=(\bigcap_{j=1}^{N-1} V_j)\cap V_N=\bigcap_{j=1}^N V_j.\]
\end{proof}

\begin{lemma}\label{equiv}
If $\{W_j\}_{j=1}^N$ are subspaces of a $k$-vector space $V$, the following are equivalent:
\begin{enumerate}[label={(\arabic*)}]
\item $(\bigcap_{j=1}^n W_j)+W_{n+1}=V$ for $1\le n<N$;
\item The natural injective linear map $V/\bigcap_{j=1}^N W_j\to\bigoplus_{j=1}^N V/W_j$ is an isomorphism;
\item There exists a basis $E$ of $V$ and pairwise disjoint subsets $\{E_j\}_{j=1}^N$ of $E$ such that $W_j=\Span(E\setminus E_j)$ for $1\le j\le N$;
\item There exist linearly independent subspaces $\{U_j\}_{j=1}^N$ of $V$ such that $V=(\bigcap_{i=1}^N W_i)\oplus\bigoplus_{i=1}^N U_i$ and $W_j=(\bigcap_{i=1}^N W_i)\oplus\bigoplus_{1\le i\le N, i\neq j}U_i$ for $1\le j\le N$;
\item $(\bigcap_{1\le j\le N, j\neq n}W_j)+W_n=V$ for $1\le n\le N$;
\item In the dual space $V^*$, the subspaces $\{W_j^\perp\}_{j=1}^N$ are linearly independent.
\end{enumerate}
If $\codim W_j<\infty$ for all $1\le j\le N$, they are also equivalent to:
\begin{enumerate}[resume,label={(\arabic*)}]
\item  $\codim(\bigcap_{j=1}^N W_j)=\sum_{j=1}^N\codim W_j$.
\end{enumerate}
\end{lemma}

\begin{remark}
If $\{W_j\}_{j=1}^N$ satisfy the equivalent conditions listed in this lemma, they are called mutually transverse. The Chinese Remainder Theorem says that comaximal ideals in a $k$-algebra are mutually transverse. Condition (6) shows that mutual transversality %can be thought of as 
is a notion dual to linear independence.
In fact, one way to prove the equivalence is passing to the dual space using the identifications $(\bigcap_{j=1}^n W_j)^\perp=\sum_{j=1}^n W_j^\perp$, $(W+W')^\perp=W^\perp\cap W'^\perp$ and $(V/W)^*=W^\perp$, and then taking advantage of the familiar equivalent characterizations of linear independence.

Although $V$ is finite-dimensional in our intended application, the proof works for any $V$. If we consider infinitely many subspaces, the obvious generalizations of the conditions in the lemma are no longer equivalent.
\end{remark}

\begin{proof}%[Proof of Lemma \ref{equiv}]
(1)$\implies$(2): If $A,B\subset V$ are subspaces, then the natural injective linear map $V/(A\cap B)\to V/A\oplus V/B$ is an isomorphism iff $V=A+B$. Thus if (1) holds, then (2) can be obtained by induction.
    
(2)$\implies$(3): Assume (2). For $1\le j\le N$, let $E'_j$ be the image of a basis of $V/W_j$ under the map $V/W_j\to\bigoplus_{j=1}^N V/W_j\cong V/\bigcap_{j=1}^N W_j$, then $\coprod_{j=1}^N E'_j$ is a basis of $V/\bigcap_{j=1}^N W_j$. Let $E_j$ be a lift of $E'_j$ to $V$, and let $E_0$ be a basis of $\bigcap_{j=1}^N W_j$, then $E:=\coprod_{j=0}^N E_j$ is a basis of $V$. By definition of $E_j$, the image of $E_j$ in $V/W_n$ is $\{0\}$ (i.e. $E_j\subset W_n$) if $n\neq j$, so $E\setminus E_n=\bigcup_{n\neq j}E_n\subset W_n$. Since $E_n$ is a basis both for $V/\Span(E\setminus E_n)$ (since $E$ is a basis of $V$) and for $V/W_n$ (by definition of $E_j$), we conclude that $W_n=\Span(E\setminus E_n)$.

(3)$\implies$(4): Take $U_j=\Span(E_j)$.

(4)$\implies$(5): Assume (4). Then $\bigcap_{j\neq n}W_j=(\bigcap_{i=1}^N W_j)\oplus U_n$, so \[\textstyle(\bigcap_{j\neq n}W_j)+W_n=(\bigcap_{i=1}^N W_i)\oplus U_n+\bigoplus_{i\neq n}U_i=V.\]

(5)$\implies$(1): Notice that $\bigcap_{j=1}^n W_j\supset
\bigcap_{1\le j\le N, j\neq n+1}W_j$ for $1\le n<N$.

(2)$\iff$(6): The dual of the injective linear map $V/\bigcap_{j=1}^N W_j\to\bigoplus_{j=1}^N V/W_j$ is canonically identified with the natural surjective map $\bigoplus_{j=1}^N W_j^\perp\to\sum_{j=1}^N W_j^\perp$.

(2)$\iff$(7): Clear.
\end{proof}

\subsection{Bound on the number of perfect powers in certain offset families of rational functions}

%\subsection{Factorization of polynomials over the separable algebraic closure}

\begin{lemma}\label{reduced}
Let $k$ be a field and $k^s$ its separable algebraic closure, so  $k^s=\ov{k}$.
\begin{enumerate}[label={(\arabic*)}]
%\item If $X$ is a reduced $k$-scheme, then $X\times_k k^s$ is reduced.
\item If $A$ is a reduced $k$-algebra, then $A\otimes_k k^s$ is reduced.
\item If $F\in k[x_1,\dots,x_n]$ is irreducible, then $F$ is square-free as a polynomial in $k^s[x_1,\dots,x_n]$.
\item If $F,G\in k[x_1,\dots,x_n]$ are non-associate irreducible polynomials, then $F$ and $G$ have no common factors in $k^s[x_1,\dots,x_n]$.
%\item If $F\in k(x_1,\dots,x_n)$ is a perfect $d$th power in $k^s(x_1,\dots,x_n)$, then $aF$ is a perfect $d$th power in $k(x_1,\dots,x_n)$ for some $a\in k^\times$.
%and if either $k$ is perfect or $\char k\nmid p$, then $F$ is a perfect $d$th power in $k[x_1,\dots,x_n]$.
\end{enumerate}
\end{lemma}
\begin{proof}
(1) This follows from \cite[Tag 030U]{Stacks}.

(2) If $F$ is irreducible over $k$, then $k[x_1,\dots,x_n]/(F)$ is reduced, so by (1), $k^s[x_1,\dots,x_n]/(F)$ is reduced, so $F$ is square-free over $k^s$.

(3) If $F,G$ are irreducible over $k$ and non-associate, then $k[x_1,\dots,x_n]/(FG)$ is reduced, so by (1), $k^s[x_1,\dots,x_n]/(FG)$ is reduced, so $F$ and $G$ have no common factors over $k^s$.

\end{proof}

\begin{lemma}\label{count}
Let $\kappa$ be a finite field and $\kappa_0$ its prime field.
Let $F\in\kappa(x_1,\dots,x_n)$ be $d$th-power-free, let $T=T_F:=\{m\in\ov{\kappa}^n\mid F(x)\equiv F(x+m)\}$ be the $\kappa_0$-subspace of $\ov{\kappa}^n$ of translations that leave $F$ invariant, and assume that $\#T<\infty$. For any finite extension $k/\kappa$, $r\in\N$ and $\{a_i\}_{i=1}^r\subset\Z$ such that $\gcd(d,a_i)=1$, let
$P$ be the collection of tuples $(m\up1,\dots,m\up r)\in k^{nr}$ such that the rational function $\prod_{i=1}^r F(x+m\up i)^{a_i}$ is a perfect $d$th power over $\ov{\kappa}$. 
Then $\#P\le C(\#k)^{n\floor{r/2}}(\#T)^{\ceil{r/2}}$, where the constant $C$ only depends on $r$ and the degree of $F$ and not on $k$.
\end{lemma}

\begin{remark}
We will not try to optimize the constant $C$. Notice that if $d\nmid(\sum a_i)(\sum b_j)$ (with $b_j$'s introduced in the proof below), then $\prod_{i=1}^r F(x+m\up i)$ is never a perfect $d$th power. However, in the case we are interested in (in the corollary that follows), $a_i=\pm1$, and $\sum a_i=0$.
\end{remark}

\begin{proof}
By Lemma \ref{reduced}, an irreducible polynomial in $\kappa[x_1,\dots,x_n]$ remains square-free %multiplicity-free
over $\ov{\kappa}=\kappa^s$, and that different irreducible polynomials remain relatively prime over $\ov{\kappa}$. 
Since $F\in\kappa(x_1,\dots,x_n)$ is $d$th-power-free, if $F=\prod_{j=1}^N f_j^{b_j}$ is the factorization of $F$ into irreducible factors over $\ov{\kappa}$, we still have $0<\abs{b_j}<d$, and in particular $d\nmid b_j$. For every $i,j$, the irreducible factor $f_j(x+m\up i)$ appears in $F(x+m\up i)^{a_i}$ with multiplicity $a_i b_j$. Since $\gcd(d,a_i)=1$ and $d\nmid b_j$, we have $d\nmid a_i b_j$. Thus, in order for $\prod_{i=1}^r F(x+m\up i)^{a_i}$ to be a perfect $d$th power, $f_j(x+m\up i)$ must also appear in $F(x+m\up{i'})$ for some $i'\neq i$, so $cf_j(x+m\up i)\equiv f_{j'}(x+m\up{i'})$ (i.e. $f_{j'}(x)\equiv f_j(x+m\up i-m\up{i'})$) for some $1\le j'\le N$ and $c\in\ov{\kappa}^\times$.

Now, for each $j$ and each tuple $m=(m\up1,\dots,m\up r)\in k^{nr}$, define an undirected graph $G_{m,j}$ with vertex set $\{1,\dots,r\}$ such that there is an edge between $i$ and $i'$ iff $f_{j'}(x)\equiv cf_j(x+m\up i-m\up{i'})$ or $cf_j(x+m\up {i'}-m\up i)$ for some $j'$ and $c$. If $m\in P$, then $G_{m,j}$ has no isolated point by the last paragraph, and it is then easy to see that it has at most $\floor{r/2}$ components. Clearly, the number of undirected graphs on $\{1,\dots,r\}$ is $2^{{r+1\choose 2}}$. Given $N$ such graphs $G_1,\dots, G_N$, we want to bound the number of tuples $m\in P$ such that $G_{m,j}=G_j$ for all $1\le j\le N$.

Let $V_j$ be the $\kappa_0$-subspace of $V:=k^n$ of translations that leave $f_j$ invariant, then $\bigcap_{j=1}^N V_j\subset T$. Choose a basis $E$ of $V$ and subsets $\{E_j\}_{j=1}^N$ as in Lemma \ref{vec sp}, so that $V_j\subset\Span(E\setminus E_j)$ for $1\le j\le N$ and $\bigcap_{j=1}^N V_j=\Span(E\setminus\bigcup_{j=1}^N E_j)$, hence $\sum_{j=1}^N\#E_j=\dim V-\dim\bigcap_{j=1}^N V_j\ge\dim V-\dim T$. An edge connecting $i$ and $i'$ in $G_{m,j}$ poses a constraint between the $E_j$-coordinates of $m\up i$ and $m\up {i'}$ under this basis; more precisely, for each $1\le j'\le N$ there are at most two possibilities for the $E_j$-coordinates of $m\up i-m\up {i'}$. Indeed, if $f_{j'}(x)\equiv cf_j(x\pm(m\up i-m\up {i'}))$ and $f_{j'}(x)\equiv c'f_j(x\pm({m'}\up i-{m'}\up {i'}))$, then $f_j(x)\equiv (c'/c)f_j(x\pm({m'}\up i-{m'}\up {i'})\mp(m\up i-{m'}\up i))$, so $c'=c$ and $({m'}\up i-{m'}\up {i'})\pm(m\up i-{m'}\up i)$ leaves $f_j$ invariant, hence it lies in $V_j\subset\Span(E\setminus E_j)$, so the $E_j$-coordinates of ${m'}\up i-{m'}\up {i'}$ are $\pm$ those of $m\up i-{m'}\up i$.

By induction, if $i$ and $i'$ lie in the same component of $G_{m,j}$, say with distance $D$, then there are at most $(2N)^D\le(2N)^r$ possibilities for the $E_j$-coordinates of $m\up i-m\up {i'}$.
Therefore, if $H\subset G_{m,j}$ is a connected component, there are at most $(\#\kappa_0)^{\#E_j}(2N)^{r(\#H-1)}$ possibilities for the $E_j$-coordinates of the $m\up i$'s with $i\in H$. If $m\in P$, $G_{m,j}$ has at most $\floor{r/2}$ components, so there are at most $(\#\kappa_0)^{\#E_j\floor{r/2}}(2N)^{r^2}$ possibilities for all the $E_j$-coordinates of $m$, and hence at most \begin{align*}
    (\#\kappa_0)^{\sum_{j=1}^N \#E_j\floor{r/2}}(2N)^{r^2 N}
    &\le (\#\kappa_0)^{(\dim V-\dim T)\floor{r/2}}(2N)^{r^2 N}\\
    &=(2N)^{r^2 N}((\#k)^n/\#T)^{\floor{r/2}}
\end{align*}possibilities for the $\bigcup_{j=1}^N E_j$-coordinates. The possibilities for the $E\setminus\bigcup_{j=1}^N E_j$-coordinates amount to $(\#T)^r$. Therefore, if we take $C=2^{{r+1\choose 2}\deg F}(2\deg F)^{r^2\deg F}$, then $\#P\le C(\#k)^{n\floor{r/2}}(\#T)^{\ceil{r/2}}$, since $N\le\deg F$.
\end{proof}
%nonabelian group: directed graph or hypergraph, #elements explodes..

\begin{corollary} \label{power count}
Fix a $d$th-power-free rational function $F\in\kappa(x_1,\dots,x_n)$ satisfying $\#T_F<\infty$,
and fix $r\in\N$. For each finite extension $k/\kappa$, let $P_k$ be the collection of tuples $(m\up 1,\dots,m\up {2r})\in k^{n2r}$ such that $\prod_{i=1}^r F(x+m\up i)\prod_{i=r+1}^{2r} F(x+m\up i)^{-1}$ is a perfect $d$th power over $\ov{\kappa}$. Then $\#P_k=O((\#k)^{nr})$ as $k$ varies.
\end{corollary}

\begin{remark}
In this case, the exponent is sharp: for any bijection $\varphi:\{1,\dots,r\}\to\{r+1,\dots,2r\}$, if $m\up {r+i}=m\up {\varphi(i)}$ for $1\le i\le r$, then $\prod_{i=1}^r F(x+m\up i)\prod_{i=r+1}^{2r} F(x+m\up i)^{-1}=1$. The number of such tuples $(m\up 1,\dots,m\up {2r})$ is asymptotic to $r!\,(\#k)^{nr}$ as $\#k\to\infty$.
\end{remark}

\subsection{Reductions of a polynomial with integer coefficients}
\begin{lemma}\label{equiv reduction}
Let $F\in\Z[x_1,\dots,x_n]$ be a polynomial, and let $x$ be the row vector $(x_1,\dots,x_n)$ of indeterminates. Then the following are equivalent:
\begin{enumerate}[label={(\arabic*)}]
\item $F$ is invariant under some nontrivial translation in $\ov{\Q}^n$, i.e. there exists $0\neq m\in\ov{\Q}^n$ such that $F(x)\equiv F(x+m)$;
\item $F$ is invariant under some nontrivial translation in $\Z^n$;
\item $F$ can be made independent of one of the indeterminates by a linear change of coordinates, i.e. there exists $A\in\GL(n,\Z)$ such that $F(xA)\in\Z[x_2,\dots,x_n]$;
\item When viewed as a morphism $\A^n_\Z\to\A^1_\Z$, $F$ factors through a linear map $\A^n_\Z\to\A^{n-1}_\Z$, i.e. there exists a integral $n\times(n-1)$ matrix $B$ and $f\in\Z[x_2,\dots,x_n]$ such that $F(x)\equiv f(xB)$;
\item For almost all prime numbers $p$, the reduction of $F$ modulo $p$ is invariant under some nontrivial translation in $\ov{\F_p}^n$.
\item For infinitely many prime numbers $p$, the reduction of $F$ modulo $p$ is invariant under some nontrivial translation in $\ov{\F_p}^n$.
\end{enumerate}
\end{lemma}
\begin{remark}
% Checking condition (1) amounts to checking whether a system of algebraic equations in $n$ variables has a nontrivial solution in an algebraically closed field, for which there is effective algorithms. % even efficient algorithms?
If the conditions are violated, (3) or (4) shows that we can reduce to a lower dimension. In fact, if we start with a homogeneous polynomial $F$ we can reduce to a homogeneous polynomial in lower dimension.
The lemma can be shown to hold for $F\in\Q(x_1,\dots,x_n)$ as well. The implication (2)$\implies$(3) fails if $\Z$ is replaced by a Dedekind domain that is not a PID.
% example: F(ax+by) with (a,b) non-principal ideal. (invariant under (b,-a)).
%should fail if $\Z$ is replaced by certain number rings. ... (Example?) ... correct setting should be $\Q$ and number fields. reduction to lower-dimension with coefficient in the field is OK, since it lie in the ring for all but finitely many primes..
\end{remark}

\begin{proof}
(1)$\implies$(2): Assume (1). Let $0\neq m=(m_1,\dots,m_n)\in\ov{\Q}^n$ be such that $F(x)\equiv F(x+m)$, we assume without loss of generality that $m_1\neq0$. Now consider $F(x+tm)-F(x)$ as a polynomial in the single indeterminate $t$. Since $F(x)\equiv F(x+m)$, by induction, every $t\in\Z$ is a root of $F(x+tm)-F(x)$, so $F(x)\equiv F(x+tm)$ since a nonzero polynomial cannot have infinitely many roots. In particular, $F(x)\equiv F(x+m/m_1)$, so we may assume that $m_1=1$ by replacing $m$ with $m/m_1$.

Let $E$ be a number field containing all the $m_i$'s. Since $F$ has coefficients in $\Z$, for any $\sigma\in\Gal(E/\Q)$, we have $F(x)\equiv F(x+\sigma(m))$, hence $F(x)\equiv F(x+\Tr_{E/\Q}(m))$. Since $m_1=1$, the first coordinate of $\Tr_{E/\Q}(m)$ is $[E:\Q]\neq0$, so we may assume that $0\neq m\in\Q^n$ by replacing $m$ with $\Tr_{E/\Q}(m)$. Let $d$ be a common denominator of the $m_i$'s, then $F(x)\equiv F(x+dm)$ and $dm\in\Z^n$.

(2)$\implies$(3): Suppose that $F$ is invariant under $0\neq m\in\Z^n$. We showed that $F(x)\equiv F(x+m)\implies F(x)\equiv F(x+tm)$ for all $t\in\ov{\Q}$, so dividing $m$ by the $\gcd(m_1,\dots,m_n)$, we may assume that $\gcd(m_1,\dots,m_n)=1$, which means that $\Z^n/\Z\cdot m$ is torsion free, hence free. Therefore $\Z^n\twoheadrightarrow\Z^n/\Z\cdot m$ splits, and if $A$ is the image of the splitting, we have $\Z^n=\Z\cdot m\oplus A\cong\Z\oplus\Z^{n-1}\cong\Z^n$, so there exists $A\in\GL(n,\Z)$ such that $m=(1,0,\dots,0)A$. We then have $F(xA)\equiv F(xA+m)\equiv F((x+(1,0,\dots,0))A)$, so the polynomial $G(x):=F(xA)$ is invariant under translation by $(1,0,\dots,0)$, so $G(x_1,x_2,\dots,x_n)-G(0,x_2,\dots,x_n)$ regarded as a polynomial in $x_1$ has all integers as its roots, and therefore must be zero. We conclude that $F(xA)\equiv G(x)\equiv G(0,x_2,\dots,x_n)\in\Z[x_2,\dots,x_n]$.

(3)$\implies$(4): Suppose that $F(xA)\equiv f(x_2,\dots,x_n)$ for some $f\in\Z[x_2,\dots,x_n]$, so $F(x)\equiv F((xA^{-1})A)\equiv f((xA^{-1})_2,\dots,(xA^{-1})_n)$, so we can take $B$ to be the last $n-1$ columns of $A^{-1}$.

(4)$\implies$(5): Suppose that there exists an integral $n\times(n-1)$ matrix $B$ and $f\in\Z[x_2,\dots,x_n]$ such that $F(x)\equiv f(xB)$. Since $B$ is a linear map from $\Q^n\to\Q^{n-1}$, the null space of $B$ is nontrivial, so one can find $0\neq m\in\Z^n$ such that $mB=0$, so $F(x)\equiv f(xB)\equiv f(xB+mB)\equiv f((x+m)B)\equiv F(x+m)$. Since $m\neq0$, the reduction of $m$ modulo $p$ is zero only for finitely many $p$ (the reductions actually lie in $(\F_p)^n$).

(5)$\implies$(6): Obvious.

(6)$\implies$(1): The conditions $F(x)\equiv F(x+m)$ and $m\neq0$ defines a subscheme $S\subset\A^n_\Z$ over $\Spec\Z$, such that the closed points in the geometric fibers $S_{\ov{\F_p}}$ or $S_{\ov{\Q}}$ 
%of $S$ over $\Spec\F_p$ or $\Spec\Q$ 
correspond to the tuples $0\neq m\in\ov{\F_p}^n$ or $\ov{\Q}^n$ such that $F(x)\equiv F(x+m)$.
By Chevalley's theorem, the image of the structural morphism $S\to\Spec\Z$ is constructible, but a constructible subset of 
%the 1-dimensional irreducible $T_0$-space (and Noetherian?)
$\Spec\Z$ either is finite or contains the generic point (and hence is cofinite).
Condition (5) says that infinitely many fibers of $S$ are nonempty, hence the image of the structural morphism contains the generic point $\Spec\Q$. Therefore, $S_\Q$ is a non-empty affine scheme and hence contains a closed point, which gives a nontrivial translation in $\ov{\Q}$ under which $F$ is invariant.
\end{proof}

\begin{lemma} \label{reduction}
Let $F\in\Z[x_1,\dots,x_n]$ be a $d$th-power-free polynomial. Then the reduction of $F$ modulo $p$ is $d$th-power-free for almost all primes $p$.
\end{lemma}
\begin{proof}
$F$ fails to be $d$th-power-free if and only if $F$ can be written as $f^d g$ such that $f$ is not constant. The coefficients of $f$ and $g$ can each be encoded in an $N$-tuple, where $N$ is the number of monomials of degrees $\le\deg F$ in $n$ indeterminates. The conditions $f^d g=F$ and that at least one of the nonconstant terms of $f$ is nonzero define a subscheme  $S\subset\A^{2N}_\Z$. If $F$ fails to be $d$th-power-free modulo infinitely many primes $p$, then $S_{\F_p}$ is nonempty for infinitely many primes, hence $S_\Q$ is nonempty (cf. proof of (6)$\implies$(1) in the previous lemma) and thus $F$ fails to be $d$th-power-free in $\ov{\Q}[x_1,\dots,x_n]$, hence in $\Q[x_1,\dots,x_n]$ (cf. proof of Lemma \ref{count}), hence in $\Z[x_1,\dots,x_n]$.
\end{proof}
Combining the previous two lemmas, we get
\begin{corollary}
Let $F\in\Z[x_1,\dots,x_n]$ be a $d$th-power-free polynomial not invariant under any nontrivial translations (in $\ov{\Q}^n$ or in $\Z^n$),
then for almost all primes $p$, the reduction of $F$ modulo $p$ satisfies the hypothesis in Lemma \ref{count} and Corollary \ref{power count} (with $T=\{0\}$).
\end{corollary}

\subsection{Degree of a projective variety and its number of points in a box}

For applications in analytic number theory, we are interested in bounding the number of points of a quasi-affine variety $X$ in a box with coordinates in a finite field. We obtain below a bound depending only on $\dim X$, $\deg X$ and the lengths of the $(\dim X)$ longest sides of the box, which is a trivial generalization of what Tao called a Schwarz--Zippel type bound in his blog post \cite{Tao}.
% Although the bound involves the longest sides and does not just scale with the size of the box, it is still compatible with the so-called Menchov--Rademacher device used to obtain Burgess bounds. observation of Michael Larsen

\begin{lemma}\label{box}
Let $k$ be an algebraically closed field, and let $X$ be a closed subvariety of $\bbP^n_k$ of codimension $\theta$ and degree $d$. If $\{B_i\}_{i=1}^n$ are subsets of $k$, we identify $B=\prod_{i=1}^n B_i$ with a subset of $\bbP^n(k)$ via the inclusions $\prod_{i=1}^n B_i\subset k^n=\A^n(k)\subset\bbP^n(k)$. If $1\le\#B_1\le\#B_2\le\dots\le\#B_n<\infty$, we have
\[\#\parens*{X(k)\cap\prod_{i=1}^n B_i}\le d\prod_{i=\theta+1}^n \#B_i=d(\#B)\prod_{i=1}^\theta (\#B_i)^{-1}.\]
\end{lemma}

\begin{remark}\label{rmk box}
In typical applications in analytic number theory, one usually takes the $B_i$'s to be intervals in some finite prime field, but it can also be applied with $B_i$ being the whole underlying set of a finite field, for example in Remark \ref{rmk tr fn}.

If $k$ is not necessarily algebraically closed, and $X$ is instead a (locally closed) subscheme of $\A^n_k$ whose irreducible components have sum of degrees $d$, the lemma still holds because we may apply the lemma to the irreducible components of the closure of $X\times_k\ov{k}$ in $\bbP^n_{\ov{k}}$ and add up the bounds. This yields Lemma \ref{ing box}.
\end{remark}
% irreducible components of closure 1-1 correspond to original

\begin{proof}
We proceed by induction on $n$. If $n=0$, then $D=0$ and $X$ must be the single point in $\A^0=\bbP^0$, so $d=1$ and both sides of the inequality are 1. If $n>0$, for each $\ov{x}\in B_n$, let $X_{\ov{x}}$ be the closed subvariety $X\cap H_{\ov{x}}$, where $H_{\ov{x}}$ is the hyperplane $\{x_n=\ov{x}x_0\}$ in $\bbP^n_k$ (here we use $(x_1,\dots,x_n)$ as the coordinates of $\A^n_k$ and $[x_0:x_1:\dots:x_n]$ as the homogeneous coordinates of $\bbP^n_k$).

If $X\subset H_{\ov{x}}$ for some $\ov{x}\in B_n$, then $X$ has the same degree $d$ as a subvariety in $H_{\ov{x}}=\bbP^{n-1}_k$. 
Therefore
\[\#\parens*{X(k)\cap\prod_{i=1}^n B_i}=\#\parens*{X_{\ov{x}}(k)\cap\prod_{i=1}^{n-1} B_i}\le d\prod_{i=\theta}^{n-1}\#B_i\le d\prod_{i=\theta+1}^n\#B_i,\]
where the first inequality is by the induction hypothesis.

If $X\not\subset H_{\ov{x}}$ for all $\ov{x}\in B_n$, then each $X_{\ov{x}}$ is a proper closed subset of the irreducible space $X$, so it has dimension $<\dim X$, hence has codimension at least $\theta$ in $H_{\ov{x}}=\bbP^{n-1}_k$.Let $Z_1,\dots,Z_s$ be the irreducible components of $X_{\ov{x}}$. By Theorem I.7.7 in \cite{Hartshorne}, $\sum_{j=1}^s\deg Z_j\le d$. Therefore
\begin{align*}
    \#\parens*{X_{\ov{x}}(k)\cap\prod_{i=1}^{n-1} B_i} 
    & \le\sum_{j=1}^s\#\parens*{Z_j(k)\cap\prod_{i=1}^{n-1} B_i}\\
    & \le\sum_{j=1}^s\deg Z_j\prod_{i=\codim Z_j+1}^{n-1}\#B_i
    \ \le\  d\prod_{i=\theta+1}^{n-1}\#B_i
\end{align*}
by the induction hypothesis, and hence
\[\#\parens*{X(k)\cap\prod_{i=1}^n B_i}=\sum_{\ov{x}\in B_n}\#\parens*{X_{\ov{x}}(k)\cap\prod_{i=1}^{n-1} B_i}\le\#B_n\cdot d\prod_{i=\theta+1}^{n-1}\#B_i=d\prod_{i=\theta+1}^n\#B_i.\]
\end{proof}

If we have a connected closed subscheme $X\subset\bbP^n_Y$ smooth over a base scheme $Y$, i.e. a family of projective schemes parametrized by $Y$, the following lemma says that all of these schemes (the fibers), possibly base extended to the algebraic closure (the geometric fibers), are equidimensional and have the same dimension and degree, and its degree equals the sum of the degrees of its irreducible components, so if the previous lemma is applied to the irreducible components (which are varieties if equipped the reduced induced scheme structure), uniform bounds are obtained.

\begin{lemma} \label{fam deg}
Let $Y$ be a scheme and let $X$ be a connected closed subscheme of $\bbP^n_Y$ smooth over $Y$. For $y\in Y$, let $X_y$ be the fiber of $X$ over $Y$, and let $X_{\ov{y}}:=X_y\times_{{\rm k}(y)}\ov{{\rm k}(y)}$. Then there exist constants $D,d\in\N_{\ge0}$ such that each $X_{\ov{y}}$ is equidimensional of dimension $D$ and degree $d$.
\end{lemma}
\begin{proof}
Since $X\to Y$ is smooth, it is flat and locally of finite presentation, and each fiber $X_y$ is smooth over ${\rm k}(y)$ and hence Cohen--Macaulay. Since $X$ is also connected, by \cite[Tag 02NM]{Stacks}, $X\to Y$ has relative dimension $D$ for some $D$, i.e. $X_y$ is equidimensional of dimension $D$ for any $y$. By ibid., Tag 02NK, $X_{\ov{y}}$ is also equidimensional of dimension $D$.
Since $X\to Y$ is flat, $X_y\subset\bbP^n_{{\rm k}(y)}$ have the same Hilbert polynomial for all $y$,
and hence the same degree $d$, for all $y$. Since the Hilbert polynomial does not change under extension of base field, all $X_{\ov{y}}\subset\bbP^n_{\ov{{\rm k}(y)}}$ have the same degree $d$.
\end{proof}

\subsection{Existence of smooth decompositions}

The next lemma assures that we can get a decomposition into smooth morphisms for very general morphisms of schemes (away from finitely many primes), and we can then apply the previous lemma to each of these smooth morphisms.
\begin{lemma} \label{smth dcmpsn}
Let $X$ be a noetherian scheme and let $\varphi:X\to Y$ be a scheme morphism of finite presentation.
Then there exist finitely many locally closed subsets $\{X_i\}_{i=1}^N$ of $X$ such that the induced morphisms $\varphi|_{X_i}:(X_i)_{\rm red}\to\ov{\varphi(X_i)}_{\rm red}$ are smooth for each $i$, and such that the image of $X\setminus\bigcup_{i=1}^N X_i$ in $\Spec\Z$ is finite.
\end{lemma}
\begin{remark}
We call such a collection $\{X_i\}_{i=1}^N$ a smooth decomposition of $\varphi$, or of $X$ relative to $Y$ (or relative to $\varphi$).
As easily seen from the proof below, the collection can be made pairwise disjoint, but we do not need that.
\end{remark}
\begin{proof}
Using noetherian induction, we need only prove the following: if $\varphi|_Z:Z\to Y$ admits a smooth decomposition for every proper closed subset $Z\subset X$ (induction hypothesis), then $\varphi$ also admits a smooth decomposition.
(Notice that a closed subscheme $Z$ of a noetherian scheme $X$ is %also noetherian and 
of finite presentation over $X$, hence over $Y$.)
If $X$ is reducible, its finitely many irreducible components $Z_j$ are proper closed subsets, so by the induction hypothesis each $\varphi|_{Z_j}$ admits a smooth decomposition, which together yield a smooth decomposition for $\varphi$.
If $X$ is irreducible, then $X_{\rm red}$ and $Y_1:=\ov{\varphi(X)}_{\rm red}$ are integral, and the induced morphism $X_{\rm red}\to Y_1$ is still of finite presentation. Therefore, if the function field $K(Y_1)$ is perfect, there exists an open dense subset $X_1\subset X$ such that $\varphi|_{X_1}:X_1\to Y_1$ is smooth \cite[Exercise 10.40]{Gortz-Wedhorn}. By the induction hypothesis, $\varphi|_{X\setminus X_1}$ admits a smooth decomposition, which together with $X_1$ gives a smooth decomposition for $\varphi$.
If the function field is not perfect, then it has nonzero characteristic, which means that the generic point $\eta\in X$ maps to a single closed point in $\Spec\Z$, so the image of $X=\ov{\{\eta\}}$ in $\Spec\Z$ is a single point, and $\varnothing$ is a smooth decomposition of $\varphi$.
\end{proof}

\iffalse
http://math.stackexchange.com/a/621813/12932
Multivariate coprime polynomials
coprime irreducibles remain coprime in field extension
hence coprime remain coprime
hence gcd remain the same

We show that irreducible remains (square-free)$^$(prime power): 
$k^s$
Let $F=\prod_{j=1}^N f_j^{b_j}$ be the factorization of $F$ in $\ov{k}[x_1,\dots,x_n]$ ($b_j>0$).
By unique factorization, $\Gal(\ov{k}/k)$ acts on the factors (up to units) $\{f_j\}_{j=1}^N$, factors in the same orbit have the same multiplicity. Assume without loss of generality that the orbit of $f_1$ is $\{\}$
Collect the factors in the orbit of say $f_j$, get say $\parens*{\prod_{j=1}^{s'}f_j}^{b_1}$. The polynomial $f:=\prod_{j=1}^{s'}f_j$ is invariant under $\Gal(\ov{k}/k)$, hence has purely inseparable coefficients. Take a power $p^t$ so that all coefficients raised to this power are in the base field. Then $\gcd(f^{p^t},F)$ is in base field and must has form $\alpha f^b$, $\alpha\in\ov{k}$. Since $F$ is irreducible, it equals the gcd, so $F=\alpha f^b$, $b\le p^t$, and $f$ is multiplicity free. Replace $f$ by $\alpha^{1/b}f$, get $F=f^b$. If $b$ is not a power of $p$, then $p^{t'}=\gcd(b,p^t)<b$, write $p^{t'}=cp^t-db$, then $f^{p^{t'}}f^{db}=f^{cp^t}$, and since $f^{db}$ and $f^{cp^t}$ are polynomials over $k$, $f^{p^{t'}}$ is also. (by induction on number of variables.) But $f^{p^{t'}}\mid F=f^b$ and has lower nonzero degree, contradicting irreducibility of $F$. Hence $b$ must be prime power. QED

(normal extension can be decomposed into inseparable extension followed by separable extension)

\fi

\section*{Acknowledgements}
I thank my collaborator Lillian Pierce %(and Roger Heath--Brown) 
for raising the original question that led to the present work. I thank my advisor Michael Larsen for his guidance, his original idea from which this paper stemmed, and numerous helpful discussions.
I thank Prof. Guocan Feng, Jianxun Hu, Lixin Liu, Zheng-an Yao and especially Yen-Mei Julia Chen, 
whose reference letters and encouragement five years ago helped me out of the dark times when my applications to PhD programs failed for two consecutive years. This paper is dedicated to them.

\end{document}